\documentclass[a4paper,11pt]{amsart}

\usepackage[utf8]{inputenc}
\usepackage[T1]{fontenc}
\usepackage{amsfonts}
\usepackage{amsthm}
\usepackage{amsmath}
\usepackage[english]{babel}
\usepackage{graphicx}
\usepackage{amscd}
\usepackage{latexsym}
\usepackage{hyperref}
\usepackage{mathrsfs}

\usepackage{mathabx,epsfig}

\usepackage[all,2cell,cmtip]{xy}
\UseTwocells

\usepackage{tikz-cd,mathtools}
\tikzset{mydescription/.style={anchor=center,fill=white}}

\usepackage{color}

\usepackage{calc}

\theoremstyle{definition}
 \newtheorem{defi}{Definition}[section]

\theoremstyle{remark}
 \newtheorem{remark}[defi]{Remark}
 
 \newtheorem{question}[defi]{Question}

\theoremstyle{plain}

\newtheorem{prop}[defi]{Proposition}

\newtheorem{lemma}[defi]{Lemma}

\newtheorem*{acknow}{Acknowledgments}

\newcommand{\zz}{\mathbb{Z}}

\newcommand{\rr}{\mathbb{R}}
\newcommand{\cc}{\mathbb{C}}
\newcommand{\hh}{\mathbb{H}}

\newcommand{\Z}[1]{\zz_{#1}}

\newcommand{\abs}[1]{\left\vert #1 \right\vert}

\newcommand{\red}{/\!\!/}

\newcommand{\redpr}[1]{\times_{\!/\!/ #1}}

\newcommand{\A}{\mathscr{A}}

\newcommand{\M}{\mathscr{M}}
\newcommand{\Mg}{\mathscr{M}^\mathfrak{g}}

\newcommand{\g}{\mathfrak{g}}
\newcommand{\Bg}{B_\mathfrak{g}(\pi)}

\newcommand{\G}{\mathscr{G}}
\newcommand{\Gc}{\mathscr{G}^c}

\newcommand{\N}{\mathscr{N}}
\newcommand{\Nc}{\mathscr{N}^c}

\newcommand{\glag}{generalized Lagrangian correspondence}

\newcommand{\Lier}{\mathcal{L}ie_\rr}
\newcommand{\Liec}{\mathcal{L}ie_\cc}

\newcommand{\Symp}{\mathcal{S}ymp}
\newcommand{\Ham}{\mathcal{H}am}
\newcommand{\Hamhat}{\widehat{\mathcal{H}am}}

\newcommand{\qHam}{q\mathcal{H}am}

\newcommand{\Cob}{\mathcal{C}ob}

\newcommand{\cqfd}{\begin{flushright}
$\Box$
\end{flushright}}

\newcommand{\MW}{Manolescu and Woodward}
\newcommand{\WW}{Wehrheim and Woodward}

\title[A 2-category of Hamiltonian manifolds]{A two-category of Hamiltonian manifolds, and a (1+1+1) field theory}

\author{Guillem Cazassus}
\address{Mathematical institute,  University of Oxford, Oxford
OX2 6GG}
\email{g.cazassus@gmail.com}

\begin{document}

\begin{abstract}
We define an extended field theory in dimensions $1+1+1$ that takes the form of a ``quasi 2-functor'' with values in a strict 2-category $\Hamhat$, defined as the ``completion of a partial 2-category'' $\Ham$, notions which we define. Our construction extends Wehrheim and Woodward's Floer Field theory, and is inspired by Manolescu and Woodward's construction of symplectic instanton homology. It can be seen, in dimensions $1+1$, as a real analog of a construction by Moore and Tachikawa.

Our construction is motivated by instanton gauge theory in dimensions 3 and 4: we expect to promote $\Hamhat$ to a (sort of) 3-category via equivariant Lagrangian Floer homology, and extend our quasi 2-functor to dimension 4, via equivariant analogues of Donaldson polynomials.
\end{abstract}

\maketitle

\tableofcontents

\section{Introduction}
\label{sec:intro}

Donaldson polynomials are invariants of smooth 4-manifolds defined by counting solutions to the anti-self dual equation. They were introduced in \cite{Donaldsonpoly} for simply connected 4-manifolds with $b_2^+$ odd and strictly greater than one, and later extended to more general 4-manifolds in \cite{kronheimer1995embedded}. They contain a lot of information and are sensitive to the smooth structure, but are also very difficult to compute. A major challenge is to understand how they undergo cut and paste operations.

Instanton homology groups are associated to some 3-manifolds, and were first introduced by Floer as a categorification of the Casson invariant. They can also be used  to define relative Donaldson invariants of 4-manifolds with boundary. In \cite{BraamDonaldsonGluing}, such relative invariants were defined (more precisely, taking values in a variation of instanton homology defined by Fukaya), and in some particular cases (avoiding reducible {connections}), some glueing formulas were obtained, making these invariants similar with a (3+1)-Topological Quantum Field Theory (TQFT). Still, whether these invariants can be interpreted as a TQFT remains an open difficult problem. See  \cite{mrowka1988local,MorganMrowkaRuberman} for work in this direction.

More generally, one might want to recast these invariants as an "extended TQFT" as in \cite{BaezDolan_cob_hyp,Lurie_cob_hyp}.  If $n$ and $d$ are integers, denote by $\Cob_{n+1+\cdots +1}$ the (weak) $d$-category with objects closed $n$-manifolds, 1-morphisms $(n+1)$-manifolds with boundary, 2-morphisms $(n+2)$-manifolds with corners, ... $d$-morphisms $(n+d)$-manifolds with corners.  An extended TQFT in this setting would (roughly) be a symmetric monoidal 4-functor from $\Cob_{0+1+1+1+1}$ (or a variation that incorporates cohomology classes) to some 4-category, such that this 4-functor applied to a closed 4-manifold (seen as a 4-morphism) essentially corresponds to its Donaldson polynomial.

In this direction, building on the Atiyah-Floer conjecture, \WW\ proposed such a functorial behaviour in dimensions 2+1+1. For this purpose they defined a 2-category that could serve as a target for a 2-functor from  $\Cob_{2+1+1}$. Their 2-category is inspired by Weinstein's symplectic category: objects are (some) symplectic manifolds, 1-morphisms are (equivalence classes of sequences of) Lagrangian correspondences, and 2-morphisms are "quilted Floer homology" classes. Unfortunately, as the symplectic manifolds and Lagrangian correspondences are singular,  Floer homology in this setting is currently undefined. Nonetheless, this construction has been implemented in slightly different settings, using nontrivial bundles, see \cite{WWfft,WWffttangles}.

In a different direction, in order to provide a mathematical definition of theories of class S appearing in string theory, Moore and Tachikawa \cite{MooreTachikawa} predicted the existence of {TQFTs} in dimensions 1+1 associating complex algebraic groups to closed 1-manifolds, and  holomorphic symplectic varieties with Hamiltonian action to 2-manifolds with boundary.

We aim to extend the functorial behaviour of Donaldson's polynomials down to dimension 1, in a way that brings together Wehrheim-Woodward and Moore-Tachikawa's constructions.  Namely, we aim to build a (symmetric monoidal) 3-category $\Ham$, and a  (symmetric monoidal) 3-functor $\Phi \colon \Cob_{1+1+1+1}\to \Ham$ such that:
\begin{itemize}
\item When restricted to the underlying 1-category $\Cob_{1+1}$, it is similar with Moore-Tachikawa's theories (see Remark~\ref{rem:MooreTachikawa} for a brief discussion about the differences).
\item When restricted to $\Cob_{2+1+1}$, seen as the endomorphism 2-category of the empty set $End(\emptyset)$ in $\Cob_{1+1+1+1}$, corresponds to \WW's Floer field theory functor.
\end{itemize}

In this paper we  build such a theory in dimension $1+1+1$. That is, we  build a 2-category $\Hamhat$, and a quasi 2-functor $\Phi \colon \Cob_{1+1+1}\to \Hamhat$, that should be  monoidal and symmetric for a monoidal and  symmetric structure on $\Hamhat$ that we will construct in \cite{monoidal}. We will explain in Section~\ref{ssec:dim_four} how we expect to extend such a functor to dimension four, by using equivariant Lagrangian Floer homology.

Our construction is inspired by  \MW 's {work} \cite{MW}, who use (open subsets of) Huebschmann and Jeffrey's extended moduli space in order to define a symplectic side of the Atiyah-Floer conjecture.

These moduli spaces satisfy a "gluing (almost) equals reduction" principle (see Proposition~\ref{prop:gluing_equals_reduction}): if $\Sigma$ has $k$ boundary components, these moduli spaces $\N(\Sigma)$ carry an $SU(2)^k$-Hamiltonian action (corresponding to constant gauge transformations on each boundary component), and if $\Sigma$ and $S$ (have {respectively} $k$ and $l$ boundary components and) are glued along $m$ components, then

\[ \N(\Sigma \cup S)\setminus C = (\N(\Sigma) \times \N(S) )\red (SU(2))^m,\]
where $C$ is a finite union of codimension 3 submanifolds.

{An important feature of the moduli spaces $\N(\Sigma)$ is that, as opposed to the Atiyah-Bott moduli space of a closed surface, if $\Sigma$ has no closed component, not only is $\N(\Sigma)$ smooth, but also \MW\ showed that it is actually possible to define Floer homology in these spaces.}

The above gluing equals reduction principle  motivates the definition of the composition of 1-morphisms (corresponding to the spaces $\N(\Sigma)$) in $\Ham$. Let us first give a naive definition of what $\Ham$ could be:

\begin{itemize}
\item objects are Lie groups,
\item 1-morphisms {from} $G$ to $G'$ are $(G\times G')$-Hamiltonian manifolds,
\item 2-morphisms: if both $M$ and $N$ are 1-morphisms {from} $G$ to $G'$, a 2-morphism $L$ from $M$ to $N$ is a $(G, G')$-Lagrangian correspondence, that is a $(G\times G')$-Lagrangian in $M^- \times N$, (see Definition~\ref{def:G-lagr}),
 \[
\xymatrix{
   G \rtwocell^{M}_{N}{L} & G'
},
\]
\item (horizontal) composition of 1-morphisms is defined as the symplectic quotient of the cartesian product: 
if $M$ is a 1-morphism form $G$ to $G'$ and $N$ a 1-morphisms form $G'$ to $G''$, define 
\[ 
M\circ_{h}^{1} N = (M\times N)\red G' ,
\]
where $G'$ acts diagonally on $M\times N$, with moment map $\mu_{diag}$ defined by 
\[ \mu_{diag} (m,n) = \mu_M(m) + \mu_N(n). \]

Both actions of $G$ and $G''$ descend to this quotient, and endow $M\circ_{h}^{1} N$ with a $G\times G'$-Hamiltonian action.
\item vertical composition of 2-morphisms is defined as {the} composition of correspondences in the usual way.
\item horizontal composition of 2-morphisms is defined as {the} quotient of the product of correspondences.
\end{itemize}

Unfortunately, this definition faces similar problems of Weinstein's symplectic category: compositions are not always well-defined. For example, $M\circ_{h}^{1} N$ can be singular if 0 is not a regular value of the moment map $\mu_{diag}$. {For this reason, the above definition is only a naive one.}

One way to remedy this problem would be to enlarge the class of geometric objects we consider (by analogy with schemes or stacks in algebraic geometry), so that these compositions are always defined. It is likely that shifted symplectic geometry could be used in order to do that. However, it would then be less obvious to define Floer homology inside such more complicated spaces. For this reason, we follow the approach of \WW \cite{WWfft}: in order to turn Weinstein's symplectic category to a honest category, they define a category where morphisms consist in equivalence classes of sequences of Lagrangian correspondences, modulo embedded compositions. Composition is then defined by concatenation (and agrees with composition of correspondences, when these are embedded). We take a similar approach at the 2-category level: we first define a ``partial 2-category'' $\Ham$, by which we essentially mean that the various kinds of compositions are only partially defined. The actual definition of a ``partial 2-category'' is slightly more complicated than that, in order to be able to produce a strict 2-category as follows (in particular, the 2-morphisms will be between \emph{sequences of} 1-morphisms). We will refer to the morphisms of $\Ham$ as simple morphisms.
 
And then, by an algebraic procedure that we call completion, we turn $\Ham$ to a strict 2-category $\Hamhat$: 
\begin{itemize}
\item $\Hamhat$ has the same objects as $\Ham$, 
\item 1-morphisms  consist in equivalence classes of sequences of simple 2-morphisms,
\item and 2-morphisms consist in equivalence classes of diagrams of simple 2-morphisms.
\end{itemize}
A subtlety  arise when defining the 2-morphism spaces: if $u,v\in hom^1(x,y)$ are 1-morphisms, the set $shom^2(u,v)$ should be independent in the choice of representatives of $u$ and $v$. For this reason, we  introduce a ``diagram axiom'' that ensures this independence of representatives, {and prove that it is satisfied in Section~\ref{ssec:diag_axiom}.}

To define our (1+1+1)-Field theory, which a priori would take the form of a 2-functor $\Cob_{1+1+1}\to \Hamhat$, we follow the same strategy of \WW : we do not associate morphisms of $\Ham$ to any morphism of $\Cob_{1+1+1}$, but only to ones that are elementary in the sense of Cerf theory, for which we know their associated moduli spaces are well-behaved (i.e. smooth, essentially). This results in a 2-functor $\Cob^{elem}_{1+1+1}\to \Hamhat$, where morphisms in $\Cob^{elem}_{1+1+1}$ are cobordisms endowed with decompositions into elementary pieces.

And then we check independence of the decompositions, using Cerf theory. If the above mentionned glueing almost equals reduction formula for $\N(\Sigma)$ would have been a strict glueing equals reduction (i.e. without the submanifolds $C$), we would have obtained a 2-functor $\Cob_{1+1+1}\to \Hamhat$. Instead, we only get a ``quasi  2-functor'', where quasi essentially refers to the presence of $C$. We believe that this shouldn't be a serious issue, as the rigid $J$-holomorphic curves involved in defining Floer homology are generically disjoint from $C$.

Notice that in \WW 's theory (in the untwisted setting at least), the main difficulties come from the fact that the moduli spaces are singular. In going to dimensions (1+1+1), this problem disappears, and the difficulty becomes checking that the resulting Floer homology groups are independent on the decompositions of surfaces. We expect that any reasonable equivariant Floer theory will be able to overcome these, with suitable ``Kirwan morphisms'' relating equivariant Floer homology to nonequivariant Floer homology of the symplectic quotient (when the latter is smooth).

In the same manner that the relevance of Fukaya categories go beyond the initial setting of their discovery (i.e. Dehn surgery in Instanton Homology \cite{BraamDonaldsonSurgery}), we expect that $\Ham$ should be an important object in symplectic topology. In the last section, we outline some possible research directions, linking to other gauge theories (Seiberg-Witten, Donaldson-Thomas), in which $\Ham$ could play a useful role.

At the 1-category level, a category similar to $\Ham$ has been introduced by Moore and Tachikawa \cite{MooreTachikawa} in the holomorphic setting, and could serve as a target for $SL(2, \cc)$ analogues of instanton homology, as introduced in \cite{AbouzaidManolescu,CoteManolescu}.

We also introduce two other partial 2-categories $\Lier$ and $\Liec$ of independent interest, similar in nature to $\Ham$ but belonging respectively to the smooth and holomorphic categories (as opposed to symplectic) that can be seen as toy models for $\Ham$, and briefly discuss some relations these partial 2-categories share with each other.

\vspace{.3cm}
\paragraph{\underline{Organization of the paper}}
In Section~\ref{sec:prelim} we set up an algebraic framework for the categories we are interested in: we introduce the notion of partial 2-categories, and construct their completion. 
In Section~\ref{sec:defHamLierLiec} we define the partial 2-category $\Ham$, as well as two analogous categories $\Lier$ and $\Liec$.
In Section~\ref{sec:moduli_spaces} we introduce the moduli spaces that are involved in our construction.
In Section~\ref{sec:constr} we construct the quasi-2-functor.
In Section~\ref{sec:future_dir} we outline some future directions that motivates the constructions in this paper.

\begin{acknow}
We thank Andriy Haydys, Dominic Joyce, Paul Kirk, Artem Kotelskiy, Jason Lotay, Andy Manion, Ciprian Manolescu, Catherine Meusburger, Mike Miller, Nicolas Orantin, Raphael Rouquier, Matt Stoffregen, Chris Woodward, Guangbo Xu and Wai-Kit Yeung for helpful conversations, in particular we thank Guangbo Xu for suggesting the relation with Seiberg-Witten theory of Section~\ref{ssec:rel_SW}. We also thank Semon Rezchikov for pointing out the work of Moore and Tachikawa, and Gregory Moore for pointing out Haydys work.
\end{acknow}

\section{Partial 2-categories}
\label{sec:prelim}

\subsection{Definition}
\label{ssec:def}

We now define partial two-categories, as two categories where the compositions are only partially defined. These can be seen as 2-category analogs of Wehrheim's categories with Cerf decompositions \cite{Wehrheimphilo}. In the next section we  associate a strict 2-category to such partial 2-categories. We denote the sets of morphisms in a partial category $shom^k$, and we  introduce other sets $\underline{hom}^k$ of "representatives of general morphisms", of which the sets of (general) morphisms $hom^k$ of the completion will be a quotient.

We start by defining a partial two-precategory. A partial two-category will satisfy an additional axiom that will be stated later in Definition~\ref{def:partial_cat}.
\vspace*{.5cm}

\textbf{Warning:} Our conventions for compositions differ from the standard one for composition of maps: if $\varphi\colon x\to y$ and $\psi\colon y\to z$ are morphisms, we denote their composition $\varphi \circ \psi\colon x\to z$.

We apologize for the length of the following definition. The opposite operations correspond to changing orientations of cobordisms, and can be safely ignored in a first reading.

\begin{defi}\label{def:partial_precat} 
A  \emph{partial two-precategory} $\mathcal{C}$ consists in:

\begin{itemize}
\item A class of \emph{objects} $Ob_{\mathcal{C}}$.

\item An involution $x\mapsto x^{op}$ on objects. $x^{op}$ is called the opposite object.

\item For each pair of objects $x,y$, a class $shom^1(x,y)$ of \emph{simple 1-morphisms}.

\item For each pair of objects $x,y$, an involutive map  $shom^1(x,y) \to shom^1(y,x)$, $\varphi \mapsto \varphi^T$, which we call \emph{adjunction}.

\item An opposite involution $shom^1(x,y) \to shom^1(x^{op},y^{op})$, $\varphi \mapsto \varphi^{op}$.

\item A \emph{partial horizontal composition}: for each triple $x,y,z$ of objects, a subset of \emph{composable 1-morphisms} 
\[comp^1(x,y,z) \subset shom^1(x,y) \times shom^1(y,z),\]
and a composition map \[ \circ_h^1 \colon comp^1(x,y,z) \to shom^1(x,z),\]
that is compatible with adjunction: if $(\varphi,\psi)\in comp^1(x,y,z)$, then $(\psi^T,\varphi^T)\in comp^1(z,y,x)$ and $(\varphi \circ_h^1\psi)^T = \psi^T\circ_h^1 \varphi^T$.

For two objects $x,y$ we define the class of \emph{representatives of general 1-morphisms} $\underline{hom}^1(x,y)$ to be the class of finite (and possibly empty if $x=y$) sequences \[\underline{\varphi} = \left(\xymatrix{ x \ar[r]^{\varphi_1} &  x_1 \ar[r]^{\varphi_2} & \cdots\ar[r]^{\varphi_k} & y   } \right),\]
with $\varphi_i\in shom^1(x_{i-1},x_{i})$, $x_0 = x$ and $x_k = y$.

Define adjunction $\underline{hom}^1(x,y) \to \underline{hom}^1(y,x)$ by: 
\[
\underline{\varphi}^T =\left(\xymatrix{ y \ar[r]^{\varphi_k^T} &  x_{k-1} \ar[r]^{\varphi_{k-1}^T} & \cdots\ar[r]^{\varphi_1^T} & x   } \right).
\]

If $\underline{\psi}\in \underline{hom}^1(x,y)$ and $\underline{\varphi}\in \underline{hom}^1(y,z)$ are such sequences, we denote $ \underline{\psi}\sharp_h^1 \underline{\varphi}\in \underline{hom}^1(x,z)$ their concatenation.

\item For any $\underline{\varphi}, \underline{\psi} \in \underline{hom}^1(x,y)$, a class of \emph{simple 2-morphisms}, denoted $shom^2(\underline{\varphi}, \underline{\psi})$.
\item Identification 2-morphisms for horizontal composition of elementary 1-morphisms: if $(\varphi, \psi) \in comp^1(x,y,z)$, {then there is a corresponding identification 2-morphism denoted:}
\[ 
I_y (\varphi, \psi)\in shom^2( (\varphi, \psi) , \varphi \circ_h^1 \psi ).
\]

\item Cyclicity: for any cyclic sequence $\underline{\varphi} \sharp_h^1 \underline{\chi} \sharp_h^1  \underline{\rho}^T  \sharp_h^1  \underline{\psi} ^T \in \underline{hom}(x,x) $, coherent\footnote{in the sense that the composition of two such identifications corresponds to the identification associated with the new reordering.}  identifications  \[shom^2(\underline{\varphi} \sharp_h^1 \underline{\chi}, \underline{\psi}  \sharp_h^1   \underline{\rho} ) \simeq shom^2(  \underline{\psi}^T \sharp_h^1 \underline{\varphi}    ,  \underline{\rho} \sharp_h^1 \underline{\chi}^T   ).\]

\[
\xymatrix{  &  y \ar[rd]^{\underline{\chi}} &  &  &  &  y  &   & \\
x  \ar[ru]^{\underline{\varphi}} \ar[rd]_{\underline{\psi}} &  \Downarrow  &  z & \simeq &  x  \ar[ru]^{\underline{\varphi}} &  \Rightarrow  &  z. \ar[lu]_{\underline{\chi}^T} \\
&  t \ar[ru]_{\underline{\rho}}&  &  &  &  t \ar[ru]_{\underline{\rho}}  \ar[lu]^{\underline{\psi}^T}&   }  
\] 

In other words, the set $shom^2(\underline{\varphi} , \underline{\psi})$ only  depends on the cyclic sequence $\underline{\varphi} \sharp^1_h \underline{\psi}^T$.

\item Opposites: $shom^2(\underline{\varphi}, \underline{\psi}) \to shom^2(\underline{\varphi}^{op}, \underline{\psi}^{op})$,  $A\mapsto A^{op}$.

\item Adjunctions: involution $shom^2(\underline{\varphi}, \underline{\psi}) \to shom^2(\underline{\psi},\underline{\varphi})$,  $A\mapsto A^T$.

We require these involutions to be compatible with the cyclicity identifications, in the sense that the following diagrams, where horizontal arrows are adjunctions and vertical arrows are cyclic identifications, should commute:

\[\xymatrix{  shom^2(\underline{\varphi} \sharp_h^1 \underline{\chi}, \underline{\psi}  \sharp_h^1   \underline{\rho} )  \ar[r]\ar[d] &  shom^2(\underline{\psi}  \sharp_h^1   \underline{\rho}, \underline{\varphi} \sharp_h^1 \underline{\chi} )  \ar[d]   \\ shom^2(  \underline{\psi}^T \sharp_h^1 \underline{\varphi}    ,  \underline{\rho} \sharp_h^1 \underline{\chi}^T   )  \ar[r]    & shom^2(  \underline{\rho} \sharp_h^1 \underline{\chi}^T ,  \underline{\psi}^T \sharp_h^1 \underline{\varphi}   ).   }\]

\item A \emph{partial vertical composition}: for each triple $\underline{\varphi},\underline{\chi},\underline{\psi}$  in $\underline{hom}^1(x,y)$, a subset of \emph{composable 2-morphisms} 
\[comp^2(\underline{\varphi},\underline{\chi},\underline{\psi}) \subset shom^2(\underline{\varphi},\underline{\chi}) \times shom^2(\underline{\chi},\underline{\psi}),\]
and a composition map \[ \circ_v^2 \colon comp^2(\underline{\varphi},\underline{\chi},\underline{\psi}) \to shom^2(\underline{\varphi},\underline{\psi}),\]
that is compatible with adjunction: if $(A,B)\in comp^2(\underline{\varphi},\underline{\chi},\underline{\psi})$, then $(B^T,A^T)\in comp^1(\underline{\psi},\underline{\chi},\underline{\varphi})$ and $(A \circ_h^1 B)^T = B^T\circ_h^1 A^T$.

If $\varphi\colon x \to y$ and $\psi\colon y \to z$ are composable 1-morphisms, we require that $I_y(\varphi, \psi)$ can be vertically composed to the right and left with any other adjacent morphism.

\end{itemize}

\end{defi}

We now define diagrams and concatenation of diagrams, which are involved in the completion.

\begin{remark} Notice that we don't have a horizontal composition of 2-morphisms in this definition. Such a composition will be defined only after completion.
\end{remark}

\begin{defi}\label{def:diagrams} Fix a partial two-precategory $\mathcal{C}$,
\begin{itemize}
\item (Representatives of general 2-morphisms) Let $\underline{\varphi}, \underline{\psi} \in \underline{hom}^1(x,y)$, define the set of representatives of general 2-morphisms $\underline{hom}^2(\underline{\varphi}, \underline{\psi})$ as the set of planar, simply connected, polygonal diagrams of simple 2-morphisms from $\underline{\varphi}$ to $\underline{\psi}$: vertices are objects, edges are simple 1-morphisms, and faces simple 2-morphisms of $\mathcal{C}$. An example of such a diagram is Diagram~\ref{diag:diagram_gen_2morphism} below.

\begin{equation}
\label{diag:diagram_gen_2morphism}
\begin{tikzcd}
 &  .\ar[r] &  . \ar[r] \ar[rrd, bend right=20, {name=A}] \ar[Rightarrow,dd, start anchor={[xshift=1.5ex, yshift=-2ex]}, end anchor={[xshift=-1ex, yshift=3ex]}] & .\ar[rd,bend left=20] 
  \ar[Rightarrow,d, start anchor={[xshift=1ex, yshift=-1ex]}, end anchor={[xshift=0ex, yshift=1.5ex]}] \ar[rd,bend left=20]
  & & \\
x\ar[ru, bend left=20]  \ar[r] \ar[rd,bend right=20] &   . \ar[Rightarrow,d, start anchor={[xshift=1ex, yshift=-1ex]}, end anchor={[xshift=-1ex, yshift=1ex]}] \ar[rd,bend left=20] &    &  \ & . \ar[r] & y\\
 &  . \ar[r]&  . \ar[r] & .\ar[ru, bend right=20]  & & \\
\end{tikzcd}
\end{equation}

Such diagrams may contain no simple 2-morphisms: if $\underline{\varphi} = \underline{\psi}$, $\mathcal{D} = \lbrace \underline{\varphi} \rbrace$ is a diagram in $\underline{hom}^2(\underline{\varphi}, \underline{\psi})$. In particular if $\underline{\varphi}\in \underline{hom}^1(x,x)$ is the empty sequence, $\mathcal{D} = \lbrace x\rbrace $ is also a diagram in $\underline{hom}^2(\underline{\varphi}, \underline{\psi})$.

\item (Vertical concatenation) Let $ \mathcal{C} \in \underline{hom}^2( \underline{\varphi}, \underline{\psi})$ and  $ \mathcal{D} \in \underline{hom}^2( \underline{\psi}, \underline{\chi})$, denote \[\mathcal{C}\sharp_v^2 \mathcal{D} \in \underline{hom}^2( \underline{\varphi}, \underline{\chi}) \] their vertical concatenation, obtained by gluing the two diagrams along $\underline{\psi}$.

\item (Horizontal concatenation) Let $ \mathcal{C} \in \underline{hom}^2( \underline{\varphi}, \underline{\psi})$ and  $ \mathcal{D} \in \underline{hom}^2( \underline{\varphi} ', \underline{\psi}')$, denote \[\mathcal{C}\sharp_h^2 \mathcal{D} \in \underline{hom}^2( \underline{\varphi}\sharp_h^1 \underline{\varphi}'  , \underline{\psi} \sharp_h^1 \underline{\psi} ' )\] their vertical concatenation, obtained by gluing the two diagrams along the common target $y$ of $\underline{\varphi}$ and $\underline{\psi}${, which is also the source of $\underline{\varphi}'$ and $\underline{\psi}'$}.

\end{itemize}
\end{defi}

One can see that the concatenations $\sharp_h^1$, $\sharp_h^2$ and $\sharp_v^2$ are associative in the strongest possible sense, and make the set of objects, $\underline{hom}^1$ and $\underline{hom}^2$ into a strict 2-category $\underline{\mathcal{C}}$, which we  refer to as the \emph{pre-completion} of $\mathcal{C}$.

\begin{defi}\label{def:comp_decomp}
Let $ \underline{\varphi} = (\varphi_0, {\ldots} , \varphi_n) \in \underline{hom}^1(x,y)$, we say that $ \underline{\varphi}'$ is a \emph{composition} of $ \underline{\varphi}$ and that  $ \underline{\varphi}$ is a \emph{decomposition} of $ \underline{\varphi}'$  if for some index $i$, $\varphi_i$ and $\varphi_{i+1}$ are composable, and $ \underline{\varphi}' = ({\ldots},\varphi_{i-1}, \varphi_i \circ_h^1  \varphi_{i+1} , \varphi_{i+2}, {\ldots} )$.
\end{defi}

\begin{defi}\label{def:partial_cat} A partial 2-precategory is a partial 2-category if the following axiom holds:

\vspace*{.3cm}
\textbf{(Diagram axiom)} Let $\underline{\varphi}_0$, $\underline{\varphi}_1$, ... , $\underline{\varphi}_k$ be a sequence of representatives of general 1-morphisms in $\underline{hom}^1(x,y)$ such that for each $i$, $\underline{\varphi}_{i+1}$ is either a composition or a decomposition of $\underline{\varphi}_i$ in the sense of Definition~\ref{def:comp_decomp}. To such a sequence is associated a diagram $\mathcal{D}\in \underline{hom}^2(\underline{\varphi}_0, \underline{\varphi}_k)$ given by patching altogether all the identification 2-morphisms or their adjoints arising from the compositions/decompositions.

The axiom requires that for any such sequence with $\underline{\varphi}_k = \underline{\varphi}_0$, the diagram $\mathcal{D}$ is an identity for $\underline{\varphi}_0$ in the following sense: for any $L\in shom^2(\underline{\psi},\underline{\varphi}_0)$, using cyclicity of simple 2-morphisms, $L$ can be composed successively with all the identification 2-morphisms or their adjoints, call $L \circ_v^2 \mathcal{D}$ the resulting 2-morphism in $shom^2(\underline{\psi},\underline{\varphi}_0)$. The diagram $\mathcal{D}$ being an identity means that  $L \circ_v^2 \mathcal{D} = L$ for any such  $\underline{\psi}$ and $L$, and also $\mathcal{D} \circ_v^2 L = L$
 for any $L\in shom^2(\underline{\varphi}_0, \underline{\psi})$, with $\mathcal{D} \circ_v^2 L$ defined analogously.
\end{defi}

\begin{remark}(Consequences of the definition)
It follows from the diagram axiom that for any composable $(\varphi,\psi)\in comp^1(x,y,z)$,  $I_y(\varphi, \psi)$ and its adjoint are inverses, in the sense that  $I_y(\varphi, \psi) \circ_v^2 I_y(\varphi, \psi)^T$ and  $I_y(\varphi, \psi)^T \circ_v^2 I_y(\varphi, \psi)$ are identities for $(\varphi, \psi)$ (resp. for $\varphi \circ^1_h \psi$) with respect to vertical composition.

Partial associativity of simple 1-morphisms also follows from this axiom: whenever all the compositions appearing are defined, one has  $(\varphi\circ_h^1 \chi)\circ_h^1 \psi = \varphi \circ_h^1 (\chi\circ_h^1 \psi ) $.
\end{remark}

\subsection{Completion}
\label{ssec:completion}
We now describe how to "complete" a partial 2-category $\mathcal{C}$ to a strict 2-category. Loosely speaking, one concatenates morphisms when they cannot be composed, as in \WW's construction of the symplectic category $\Symp^\#$ \cite{WWfft}.

\begin{defi}(Completion of a partial 2-category)\label{def:completion} 
Let $\mathcal{C}$ be a partial 2-category. The following construction defines a strict 2-category $\widehat{\mathcal{C}}$, called completion of $\mathcal{C}$.

\begin{itemize}
\item Objects of $\widehat{\mathcal{C}}$ are the same as the objects of $\mathcal{C}$.
\item Given two object $x$ and $y$, the set of $1$-morphisms $hom^1(x,y)$ is defined as the quotient of $\underline{hom}^1(x,y)$ by the relation generated by compositions: we identify $\underline{\varphi}$ and  $\underline{\varphi}'$ if $\underline{\varphi}'$ is a composition of $\underline{\varphi}$ in the sense of Definition~\ref{def:comp_decomp}. Moreover, if $\varphi\in shom^1(x,x)$ is an identity for $x$ in the sense that it can be composed to the left and right with any adjacent simple 1-morphism $\psi$, and $\varphi\circ^1_h \psi = \psi$ or $\psi\circ^1_h \varphi = \psi$; then we identify such an identity with the empty sequence.

\item To define the spaces of 2-morphisms we first define a set  $hom^2(\underline{\varphi}, \underline{\psi})$ for representatives $\underline{\varphi}, \underline{\psi}\in \underline{hom}^1(x,y)$, and use the diagram axiom to define $hom^2([\underline{\varphi}], [\underline{\psi}])$, with $[\underline{\varphi}], [\underline{\psi}]$ the equivalence classes in $hom^1(x,y)$.

Let $hom^2(\underline{\varphi}, \underline{\psi})$ be the quotient of $\underline{hom}^2(\underline{\varphi}, \underline{\psi})$ by the following relation: if $\mathcal{D}$ is a diagram such that two faces $A$ and $B$ have a connected intersection $\underline{\chi}\in \underline{hom}^1(x,y)$, by cyclicity we can assume that $A\in shom^2(\underline{\alpha},\underline{\chi})$ and $B\in shom^2(\underline{\chi},\underline{\beta})$, for some $\underline{\alpha}$ and $\underline{\beta}$. If $A$ and $B$ are composable, then we identify $\mathcal{D}$ with the diagram obtained by removing the edge $\underline{\chi}$, and merging the two faces $A$ and $B$ to a single one $A\circ_v^2 B$. We also identify the empty diagrams with identities for $\circ^2_v$. These identifications generate the equivalence relation.

We now define $hom^2([\underline{\varphi}], [\underline{\psi}])$: pick two representatives $\underline{\varphi}, \underline{\varphi}' \in [\underline{\varphi}]$, which by assumption can be joined by a sequence of compositions and decompositions $\underline{\varphi}_0 = \underline{\varphi}$, $\underline{\varphi}_1$, ... , $\underline{\varphi}_k = \underline{\varphi}'$. Such a sequence might not be unique, pick any other such sequence $\tilde{\underline{\varphi}}_0 = \underline{\varphi}$, $\tilde{\underline{\varphi}}_1$, ... , $\tilde{\underline{\varphi}}_l = \underline{\varphi}'$. To these two sequences are associated two diagrams $\mathcal{D}$, $\tilde{\mathcal{D}}$  of identification 2-morphisms, and vertical concatenation defines two maps
\begin{align*}
 m&\colon hom^2(\underline{\varphi}, \underline{\psi})\to hom^2(\underline{\varphi}', \underline{\psi}),\ [\mathcal{A}]\mapsto [\mathcal{D}\sharp^2_v \mathcal{A}], \\
\tilde{m}&\colon hom^2(\underline{\varphi}, \underline{\psi})\to hom^2(\underline{\varphi}', \underline{\psi}),\ [\mathcal{A}]\mapsto [\mathcal{D}'\sharp^2_v \mathcal{A}]. \\
\end{align*}

The diagram axiom applied to the sequence \[\underline{\varphi}_0 , \underline{\varphi}_1, {\ldots} , \underline{\varphi}_k, \underline{\varphi}_{k-1}, {\ldots} , \underline{\varphi}_0\] shows that the map $m$ is invertible, and that its inverse is given by $[\mathcal{A}]\mapsto [\mathcal{D}^T\sharp^2_v \mathcal{A}]$.

Applying now the diagram axiom to the sequence \[\tilde{\underline{\varphi}}_0 , \tilde{\underline{\varphi}}_1, {\ldots} , \tilde{\underline{\varphi}}_l, \underline{\varphi}_{k-1}, {\ldots} , \underline{\varphi}_0\] shows that the map $\tilde{m} \circ m^{-1}$ is the identity. In other words, for any pair of representatives $\underline{\varphi}, \underline{\varphi}' \in [\underline{\varphi}]$, the sets $hom^2(\underline{\varphi}, \underline{\psi})$ and $hom^2(\underline{\varphi}', \underline{\psi})$ are canonically identified. One can similarly prove that two equivalent choices for $\underline{\psi}$ induce canonical identifications. It follows that $hom^2(\underline{\varphi}, \underline{\psi})$ only depends on the classes $[\underline{\varphi}]$ and $[\underline{\psi}])$, and can be denoted $hom^2([\underline{\varphi}], [\underline{\psi}])$.

\item Identities. Let $x$ be an object, the identity 1-morphism associated to $x$ is defined as the class of the empty sequence.

Let $[\underline{\varphi}]$ be a 1-morphism, the identity 2-morphism is defined as the class of the diagram with no simple 2-morphisms, consisting only in $\underline{\varphi}$.

\item The three concatenations $\sharp_h^1$,  $\sharp_h^2$,  $\sharp_v^2$  pass to the quotient and define respectively  composition maps:
\begin{align*}
\circ_h^1 &\colon hom^1(x,y)\times hom^1(y,z)\to hom^1(x,z), \\
\circ_h^2 &\colon hom^2([\underline{\varphi}],[\underline{\psi}])\times hom^2([\underline{\varphi}'],[\underline{\psi}'])\to hom^2([\underline{\varphi}] \circ_h^1 [\underline{\varphi}'],[\underline{\psi}] \circ_h^1 [\underline{\psi}'] ), \\
\circ_v^2 &\colon hom^2([\underline{\varphi}],[\underline{\chi}])\times hom^2([\underline{\chi}],[\underline{\psi}])\to hom^2([\underline{\varphi}],[\underline{\psi}]), \\
\end{align*}
that satisfy the associativity properties of a strict 2-category.
\end{itemize}

\end{defi}

\begin{remark}(Self-criticism of the construction of completion)
\label{rem:self_criticism}
In the equivalence relation on $\underline{hom}^2(\underline{\varphi}, \underline{\psi})$ we only consider the case where the intersection of the two faces $A$ and $B$ is connected, but it can happen that faces intersect along a disconnected set of edges. If each of the corresponding compositions are allowable, it would be natural to also allow such identifications, however that would lead to more general notions of simple two morphisms in the definition of a partial 2-category, as the corresponding faces might not be polygons anymore, but non-simply connected regions of the plane. We won't do that in this paper, for sake of simplicity. 
\end{remark}

\begin{remark}(Completion as a solution to a universal problem)\label{rem:univ_pb_completion}
There is probably a more intrinsic way of defining completion as a solution to a universal problem. After having a suitable definition of a partial 2-functor, a completion of $\mathcal{C}$ could consist in a pair $(\widehat{\mathcal{C}},f)$ of a strict 2-category $\widehat{\mathcal{C}}$ together with a ``partial 2-functor'' $f\colon \mathcal{C} \to \widehat{\mathcal{C}}$ such that any other partial 2-functor from $\mathcal{C}$ to any other strict 2-category factors through $f$ in an essentially unique way. It would be interesting to work out such a definition in more detail, and compare it to the definition we give, or the possibly more natural one we just alluded in Remark~\ref{rem:self_criticism}.
\end{remark}

\section{Definition of $\Ham$, $\Lier$ and $\Liec$}
\label{sec:defHamLierLiec}

We shall now define the partial 2-category $\Ham$. Although this category will be our main object of interest, we also introduce two simpler and closely related categories $\Lier$ and $\Liec$, which one could view as toy models for $\Ham$. The fact that these categories satisfy the diagram axiom is nontrivial and is proved in Section~\ref{ssec:diag_axiom}.

\subsection{Definition of the partial 2-precategories $\Lier$ and $\Liec$}
\label{ssec:def_Lier_Liec}

\begin{defi}The following defines a partial 2-category $\Lier$:
\label{def:Lier}
\begin{itemize}
\item Objects are real Lie groups,

\item The opposite $G^{op}$ of a group $G$ is the group itself endowed with the opposite multiplication \[ g \cdot_{op} h = h \cdot g.\]

\item The simple 1-morphisms from $G$ to $G'$ consist in real smooth manifolds endowed with a left action of $G$ and a right action of $G'$. These two actions should commute.

\item Adjunction of simple 1-morphisms: If $M\in shom^1(G,G')$, its adjoint $M^T \in shom^1(G',G)$ consists in the same underlying manifold, with the new action defined as 
\[ 
g'\cdot m\cdot g = g^{-1} \cdot m \cdot (g')^{-1} ,
\]
for $g\in G$ and $g' \in G'$.

\item Opposites $shom^1(G,G')\to shom^1(G^{op},G'^{op})$
\[ 
g\cdot_{op} m\cdot_{op} g' = g^{-1} \cdot m \cdot (g')^{-1} ,\ g\in G, m\in M, g'\in G' .
\]
\item Horizontal composition of simple 1-morphisms is defined as a "covariant product". We will say that $M_{01}\in shom^1(G_0,G_1)$ and $M_{12}\in shom^1(G_1,G_2)$ are  composable if the action of $G_1$ on $ M_{01}\times M_{12}$ defined by $g\cdot (m,m') = (mg^{-1}, g m')$ is free and proper.

When this is the case, we will define the composition $M_{01} \circ^1_h M_{12}$ as the quotient 
\[
M_{01} \times_{G_1}M_{12} =( M_{01} \times M_{12})/_{G_1}
\] 
for this action. Since they commute with the $G_1$-action, the actions of $G_0$ and $G_2$ pass to this quotient.

\item Simple 2-morphisms. Let 
\begin{align*}
\underline{M} &= (M_{01}, M_{12}, {\ldots} , M_{(k-1)k}), \text{ and} \\ \underline{N} &= (N_{01}, N_{12}, {\ldots} , N_{(l-1)l})
\end{align*}
be in $\underline{hom}^1(G,G')$,  with 
\[
M_{i(i+1)} \colon G_i \to G_{i+1}\text{,  and }N_{j(j+1)} \colon H_j \to H_{j+1}.
\]
Denote by $\prod{\underline{M}}$ and $\prod{\underline{N}}$ the product of all the 1-morphisms appearing respectively in $\underline{M}$ and $\underline{N}$.  A simple 2-morphism from $\underline{M}$ to $\underline{N}$ is a submanifold of the product 
\[
\prod \underline{M} \times \prod \underline{N},
\]
 which is invariant by the action of all the groups $G_i$ and $H_j$, where $G_i$ acts on $M_{(i-1)i} \times M_{i(i+1)}$ by 
 \[
 g(m,m') = (mg^{-1}, gm')
 \]
 and acts trivially on the other factors, except for the two extremal groups $G$ and $G'$ that act on $M_{01}\times N_{01}$ and $M_{(k-1)k} \times N_{(l-1)l}$ respectively by 
\begin{align*}
 g(m,m') &= (gm, gm')\text{, and} \\
 g(m,m') &= (mg^{-1}, m'g^{-1}).
\end{align*}
 We will call such submanifolds \emph{multi-correspondences}.

\item Identification 2-morphisms. If $(M_{01}, M_{12}) \in comp^1(G_0,G_1,G_2)$ are composable 1-morphisms, the identification 2-morphism \[I_{G_1}(M_{01}, M_{12}) \in shom^2( (M_{01}, M_{12}), M_{01}\times_{G_1} M_{12} )\] is given by the graph of the projection $M_{01}\times M_{12} \to M_{01}\times_{G_1} M_{12}$.

\item Cyclicity and adjunction for 2-morphisms are the obvious identifications.

\item Partial vertical composition is defined as compositions of correspondences: If $ \underline{M}, \underline{N}, \underline{P}\in \underline{hom}^1(G,G')$,  $A\in shom^2(\underline{M}, \underline{N})$, and  $B\in shom^2(\underline{N}, \underline{P})$, we say that $A$ and $B$ are composable if $A\times \prod \underline{P}$ and $\prod \underline{M} \times B$ intersect transversally in $\prod \underline{M}\times \prod \underline{N}\times \prod \underline{P} $, and if the projection \[\prod \underline{M}\times \prod \underline{N}\times \prod \underline{P}  \to \prod \underline{M}\times  \prod \underline{P} \] is an embedding when restricted to this intersection. When this is the case, the composition $A \circ_v^2 B$ is defined as the image of this intersection by this projection.
\end{itemize}
\end{defi}
We will prove that $\Lier$ satisfies the diagram axiom in Section~\ref{ssec:diag_axiom}.

\begin{remark}(Relation with the category of Lie groups)
\label{rem:categ_lie_gr} The category of Lie groups embeds in $\Lier$ in a natural way. Let $f\colon G\to G'$ be a group morphism, then $M_f = G'$, endowed with the bi-action of $G$ and $G'$ defined by 
\[
g\cdot m \cdot g' = f(g) m (g'),\ g\in G, m\in M_f, g'\in G'
\] 
is an elementary 1-morphism of $\Lier$. Moreover the composition of group morphisms agrees with the horizontal composition of the $M_f$'s. One can therefore think of  $\Lier$ as an enlargement of the category of Lie groups. Furthermore,  $M_{id_G}$ plays the role of the identity for $G$.
\end{remark}

Before defining our main partial 2-category of interest, we find interesting to point out that it is already possible to produce a (sort of) (1+1)-field theory with values in the underlying 1-category of $\Lier$.

Observe first that at the level of simple morphisms, the cartesian product endows $\Lier$ with a sort of Frobenius algebra structure. For any object $G$, we can define a unit/counit, an identity, and a product/coproduct:

\begin{itemize}
\item Unit and counit. Let $e_G \in shom^1(1,G)$ consist in the point, endowed with the trivial action. The counit $(e_G)^T \in shom^1(G,1)$ is its adjoint. Composing to the right with $e_G$ corresponds to modding out by $G$.

\item Identity. For $id_G \in shom^1(G,G)$ one can take $M_{Id_G}$, with $Id_G$ the identity group morphism.

\item Product and coproduct. They can both be obtained from a simple 1-morphism $M\in shom^1(G\times G\times G, 1)$ we define now. Let  $G_0$, $G_1$, and $G_2$ stand for three copies of the same group $G$. Their product $\widetilde{M} = G_0\times G_1\times G_2$ admits three left actions, each group $G_i$ acts on $\widetilde{M}$ by left multiplication on its corresponding factor. Take then $M$ to be the quotient of $\widetilde{M}$ by $G$, where $G$ acts by simultaneous right multiplication on each factor. The three actions descend to $M$. Moreover the slice $G_0\times G_1\times \lbrace e \rbrace \subset \widetilde{M}$ identifies $M$ with $G_0\times G_1$, and under this identification the three actions become, with $g_i\in G_i$, $a\in A$ and $b\in B$:
\begin{align*}
 g_0.(a,b) =& (g_0 a,b), \\
 g_1.(a,b) =& (a,g_1 b), \\
 g_2.(a,b) =& (a g_2^{-1},b g_2^{-1}). \\
\end{align*}

By turning the $G_2$-action to a right action (i.e. acting by $g_2^{-1}$) one can think of $M$ as a product, i.e. in $shom^1(G_0 \times G_1, G_2)$. Doing likewise with $G_1$, one can think of $M$ as a coproduct, i.e. in $shom^1(G_0,  G_1 \times G_2)$.
\end{itemize}
One can check that these are indeed associative and coassociative, that $e_G$ and $(e_G)^T$ are indeed units and counits. To any surface with boundary, one can then associate a generalized 1-morphism in $\widehat{\Lier}$, after taking a pair of pants decomposition, and associating a copy of $M$ to each pair of pants.

One can for example produce a 1-morphism in $shom^1( G, 1)$ that corresponds to  the punctured torus, by contracting the coproduct with the identity, seen as in $shom^1(G\times G, 1)$. This corresponds to $G$ itself, endowed with its conjugation action. Interestingly, its cotangent bundle ends up being similar with the extended moduli space that we will associate to the punctured torus.

\begin{defi}\label{def:Liec} One can define a complex analogue $\Liec$ of $\Lier$, by taking complex Lie groups as objects, complex manifolds as 1-morphisms, and complex multi-correspondences as 2-morphisms.
\end{defi}

\subsection{Definition of the partial 2-precategory $\Ham$}
\label{ssec:def_precat}

We first recall some standard facts about Hamiltonian actions that will be relevant to the construction of $\Ham$, for the reader's convenience and to set some notation conventions.

\begin{defi}\label{def:ham_action}(Hamiltonian manifold)
Let $G$ be a Lie group. A (left) Hamiltonian $G$-manifold $(M,\omega,\mu)$ is a symplectic manifold  $(M,\omega)$  endowed with a left $G$-action by symplectomorphisms, induced by a moment map $\mu \colon M\to \mathfrak{g}^*$. The moment map is $G$-equivariant with respect to this action and  the coadjoint representation on $\mathfrak{g}^*$, and satisfies the following equation: 
\[ 
\iota_{X_\eta} \omega = d \langle \mu, \eta\rangle ,
\]
for each $\eta \in \mathfrak{g}$, where $X_\eta$ stands for the vector field on $M$ induced by the infinitesimal action, i.e. 
\[
X_\eta (m)= \frac{d}{dt}_{|t=0} (e^{t\eta} m).
\]
In other words, $X_\eta$ is the symplectic gradient of the function $\langle \mu, \eta\rangle$.

A right action will be said Hamiltonian with moment $\mu$ if the associated left action is Hamiltonian with moment $-\mu$.

\end{defi}

\begin{remark}If $G$ is connected, the moment map determines the action. If $G$ is discrete, a Hamiltonian action is just an action by symplectomorphisms.
\end{remark}

\begin{defi}\label{def:Weinstein_corresp}
Weinstein observed in \cite{Weinstein_sg} that the data of both the action and the moment map can be conveniently packaged as a Lagrangian submanifold $\Lambda_G(M)\subset T^*G\times M^-\times M$, defined as:
\[
 \Lambda_G(M) = \lbrace ((q,p),m,m') : m' = q.m,\ R_{g^{-1}}^*p = \mu(m)  \rbrace.
  \]
When $M = T^*X$ is a cotangent bundle and the action and the moment are the ones canonically {induced} from a smooth action on the base $X$, $\Lambda_G(M)$ corresponds to the conormal bundle of the graph of the action  
\[
\Gamma_G(X) \subset G\times X\times X ,
\]
where one identifies $T^* X$ with $(T^* X)^-$ via $(q,p)\mapsto (q,-p)$. 
\end{defi}

\begin{defi}\label{def:reduction}(Symplectic quotient) 
If $(M,\omega,\mu)$ is a Hamiltonian manifold, its \emph{symplectic quotient} (or \emph{reduction}) is defined as 
\[ 
M\red G  = \mu^{-1}(0) / G . 
\]
When $0$ is a regular value for $\mu$, and $G$ acts freely and properly on $\mu^{-1}(0)$, $M\red G $ is also a symplectic manifold. In this case, we will say that the action is \emph{regular}. 
\end{defi}

\begin{defi}\label{def:canon_corresp}(Canonical Lagrangian correspondence between $M$ and $M\red G$) If the action is regular in the sense of Definition~\ref{def:canon_corresp}, the image of the map $\iota \times \pi  \colon \mu^{-1}(0) \to  M^- \times M\red G$ is a Lagrangian correspondence, where $\iota$ and $\pi$ stand respectively for the inclusion and the projection.

\end{defi}
\begin{remark}\label{rem:residual_action}(Induced action on the quotient by a normal subgroup) If $(M,\omega,\mu)$ is a Hamiltonian $G$-manifold, and if $H\subset G$ is a normal subgroup, then $M$ is in particular a Hamiltonian $H$-manifold, with moment map obtained by composing $\mu$ with the dual of the inclusion of the Lie algebras. If furthermore the $H$-action is regular as in  Definition~\ref{def:reduction}, then $M\red H$ carries a residual Hamiltonian action of $G/H$.

In particular, if $M$ is a symplectic manifold endowed with two commuting Hamiltonian actions of two groups $G$ and $G'$ (that is, a Hamiltonian $G\times G'$-action), and if the action of $G$ is regular, then $M\red G$ is a Hamiltonian $G'$-manifold.

\end{remark}

\begin{remark}\label{rem:opposite_symplectic_str}(Action on $M^-$) If $(M,\omega)$ is a $G$-Hamiltonian manifold with moment map $\mu$, then $M^- = (M,-\omega)$ endowed  with the same action is also Hamiltonian, with moment map $-\mu$.
\end{remark}

\begin{defi}\label{def:G-lagr}($G$-Lagrangian)
A $G$-Lagrangian of a Hamiltonian $G$-manifold $M$ is a Lagrangian submanifold $L\subset M$ that is both contained in the zero level $\mu^{-1}(0)$, and  $G$-invariant. When the $G$-action is regular, the $G$-Lagrangians of $M$ are in one-to-one correspondence with the Lagrangians on $M\red G$ (Though a Lagrangian in $M$ need not be a $G$-Lagrangian to induce a Lagrangian on $M\red G$). 

\end{defi}

We now define $\Ham$.

\begin{defi}(The partial 2-category $\Ham$) The following construction defines a  partial 2-category:
\label{def:Ham}
\begin{itemize}
\item Objects are real Lie groups, the opposite map sends $G$ to $G^{op}$ (with the opposite group structure).

\item The simple 1-morphisms from $G$ to $G'$ are the symplectic manifolds endowed with commuting Hamiltonian left $G$-action and right $G'$-action, with respective moment maps $\mu_G$ and $\mu_{G'}$. Equivalently, a Hamiltonian left $(G\times G')$-action with moment map 
\[
\mu = (\mu_G, -\mu_{G'}) \colon M\to Lie(G)^* \times Lie(G')^* .
\]
The moment maps are part of the data.

\item The opposite identification $shom^1(G,H) \to shom^1(G^{op},H^{op})$ is given by taking the same symplectic manifold, and acting through $g^{-1}$ instead of $g$. 

\item The adjunction identification $shom^1(G,H) \to shom^1(H,G)$ is given by reversing the symplectic structure and acting through $g^{-1}$ instead of $g$.

\item Horizontal composition of simple 1-morphisms is defined as a diagonal symplectic reduction of the product. Let two simple 1-morphisms $M_{01}\in shom^1(G_0,G_1)$ and $M_{12}\in shom^1(G_1,G_2)$:
\[
\xymatrix{ G_0 \ar[r]^{M_{01}} &  G_1 \ar[r]^{M_{12}} &  G_2   }.
\] 

Endow $M_{01}\times M_{12}$ with the diagonal action of $G_1$ {given by} \[g\cdot (m,m') = (mg^{-1},gm'),\] which is Hamiltonian with respect to the moment map 
\[
\mu^{diag}_{G_1}(m_{01},m_{12}) =-\mu_{G_1}(m_{01}) + \mu_{G_1}(m_{12}). 
\]

We will say that $M_{01}$ and $M_{12}$ are composable if this action is regular. 
In this case, we define the horizontal composition of $M_{01}$ and $M_{12}$ as the symplectic reduction of $M_{01}\times M_{12}$ for this action, and denote it 

\[
M_{01} \redpr{G_1}M_{12} = (\mu^{diag}_{G_1})^{-1}(0) /_{G_1},
\] 
or sometimes $M_{01} \redpr{}M_{12}$ when $G_1$ is implicit.

It is endowed with its reduced symplectic structure, and the two actions  of $G_0$ and $G_2$ (and their moment maps) descend to the quotient: $M_{01} \redpr{G_1}M_{12}$ is then a simple morphism from $G_0$ to $G_2$.

Notice that this construction depends on the moment maps, and not just on the Hamiltonian actions.

\item Simple 2-morphisms:
 with
\begin{align*} \underline{M} &=  \left(\xymatrix{ G \ar[r]^{M_{01}} &  G_1 \ar[r]^{M_{12}} & \cdots \ar[r]^{M_{(k-1)k}} & G'}\right)\text{, and}\\
\underline{N} &= \left(\xymatrix{ G \ar[r]^{N_{01}} &  H_1 \ar[r]^{N_{12}} & \cdots \ar[r]^{N_{(l-1)l}} & G'   }\right) ,
\end{align*}
we define $shom^2(\underline{M},\underline{N})$ as the set of $K$-Lagrangians of $P$, where
\begin{align*} P =& \prod_i{M_{i(i+1)} ^-} \times \prod_j{N_{j(j+1)}}\text{, and}\\
K =& G \times \left( \prod_{i=1}^{k-1}{G_i} \right)\times G' \times \left(\prod_{j=1}^{l-1}{H_j}\right),
\end{align*}
where each factor of $K$ acts diagonally  with the diagonal moment map on the two symplectic manifolds associated to it (if the symplectic manifold comes with its opposite symplectic form, one takes the opposite moment map, in accordance with Remark~\ref{rem:opposite_symplectic_str}). We will call such submanifolds \emph{\glag}.

\item The cyclicity isomorphisms are the obvious ones, as for the opposites and adjunction of 2-morphisms.

\item The identification 2-morphisms  \[I_{G_1}(M_{01}, M_{12}) \in shom^2( (M_{01}, M_{12}), M_{01} \redpr{G_1} M_{12}) \] are defined as the canonical Lagrangian correspondence between $M_{01} \times M_{12}$ and its reduction, defined in Definition~\ref{def:canon_corresp}.

\item Vertical composition of simple 2-morphisms is defined as  composition of Lagrangian correspondences. Let $\underline{M},\underline{N},\underline{P}$ be in $\underline{hom}^1(G,G')$, $L\in shom^2(\underline{M},\underline{N})$ and $L'\in shom^2(\underline{N},\underline{P})$. Say that $L$ and $L'$ are composable if $L\times \prod {\underline{P}}$ and $\prod{\underline{M}} \times L'$ intersect transversely in
$\prod{\underline{M}} \times\prod{\underline{N}} \times\prod{\underline{P}} $, and if the projection of this intersection to  $\prod{\underline{M}} \times\prod{\underline{P}} $ is an embedding. If this is the case, the vertical composition is defined as being the image of this embedding, and is an element in $shom^2(\underline{M},\underline{P})$.

\end{itemize}
\end{defi}

\begin{remark}\label{rem:rel_Lier_Ham}(Relation between $\Lier$ and $\Ham$) 
In \cite[Sec.~5]{Weinstein_Symplectic_categories},  Weinstein shows that there is a functor from the category of manifolds and smooth maps to his symplectic category. One can extend such a functor to $\Lier$: there is a "partial 2-functor" $\Lier \to \Ham$. In short, it is defined at the level of simple morphisms by sending a group to itself, a 1-morphism to its cotangent bundle, and a 2-morphism to its conormal bundle:
\begin{align*}
 G &\mapsto G ,\\
 M &\mapsto T^*M ,\\
 C  &\mapsto N^*_C M .
\end{align*}

We will refer to it as the cotangent 2-functor. 

Indeed, a smooth action of $G$ on $M$ lifts to an action on $T^* M$ defined by 
\[
g\cdot (q,p) = (gq, p\circ (D_q L_g)^{-1}), 
\]
with $g\in G$, $(q,p)\in T^*M$.

It follows from Cartan's formula that this action is Hamiltonian, with moment map $\mu \colon T^*M \to \mathfrak{g}^*$ defined by \[\mu(q,p) \cdot \xi = p(X_\xi(q)). \]

Moreover, the horizontal compositions of 1-morphisms are compatible: if two simple 1-morphisms  \[\xymatrix{ G_0 \ar[r]^{M_{01}} &  G_1 \ar[r]^{M_{12}} &  G_2   }\]
are composable in $\Lier$, then the same holds in $\Ham$ for
\[
\xymatrix{ G_0 \ar[r]^{T^*M_{01}} &  G_1 \ar[r]^{T^*M_{12}} &  G_2   },
\]
one has $T^* (M_{01}\times_{G_1} M_{12}) = T^*M_{01} \redpr{G_1}T^*M_{12}$, and the conormal bundle of the identification 2-morphism $I_{G_1}(M_{01}, M_{12})$ is $I_{G_1}(T^*M_{01}, T^*M_{12})$.

If $C\subset M\times N$ is a smooth correspondence, then its conormal bundle is a Lagangian submanifold $N^*_C M \subset T^*M\times T^*N$. But $T^*M$ is symplectomorphic to its opposite symplectic structure via the map $(q,p)\mapsto (q, -p)$, so applying this identification, one can think of it as a Lagrangian correspondence  $N^*_C \subset (T^*M)^-\times T^*N$.

If moreover $C$ is invariant for a $G$-action, then $N^*_C$ will be a $G$-Lagrangian. If two simple 2-morphisms are vertically composable in $\Lier$, then the same is true for their conormal bundles, and the conormal bundle of their composition agrees with the composition of the conormal bundles.
\end{remark}

\begin{remark}Following the arborealization program of Nadler \cite{Nadler_Arboreal_singularities}, it is maybe possible to extend the class of 1-morphisms in $\Lier$ to manifolds with certain type of singularities (with extra local structure), and extend the cotangent functor by sending a singular manifold to a Weinstein manifold with Lagrangian skeleton being the singular manifold.
\end{remark}

\begin{remark} We can use this functor to transport structure from $\Lier$ to $\Ham$, such as  identities, products and coproducts, units and co-units ... that would lead to a $(1+1)$-theory taking values in $\Ham$. The 2-functor we will construct in Section~\ref{sec:constr}, {although} close to, will not correspond to that.
\end{remark}

\begin{remark}\label{rem:identities_ham}(Identities) For any Lie group $G$, its cotangent bundle $T^*G$, endowed with left and right pullbacks, plays the role of an identity (i.e. can be composed with any other morphism, and the composition is the same morphism). It is  the image of the identity of $G$ in $\Lier$ by the cotangent functor.

If $M\colon G \to G'$ is a simple 1-morphism, the diagonal $\Delta_M\in shom^2(M,M)$ plays the role of an identity. However, for a general 1-morphism representative $\underline{M} = (M_0, M_1, {\ldots}, M_k) \in \underline{hom}^1(G,G')$, the product of the diagonal of all its factors is \emph{not} an identity, since it is not an element in $shom^2(\underline{M},\underline{M})$, see Remark~\ref{rem:horiz_2comp_ham}.

One can nevertheless find length two identity representatives as follows: let 
\[
\widehat{\underline{M}} = (M_0, T^* G_1, M_1, T^* G_2,  {\ldots}, T^* G_{k-1}, M_k) \in \underline{hom}^1(G,G'),
\]
then Weinstein's correspondence (see Definition~\ref{def:Weinstein_corresp}) for the action of  
\[
K =G_1\times \cdots \times G_{k-1}
\] 
on $\prod \underline{M}$ defines a simple 2-morphism 
\[
\Lambda (\prod \underline{M}) \in shom^2(\underline{M},\widehat{\underline{M}}),
\]
and \[\left(\Lambda_K (\prod \underline{M}), \Lambda_K (\prod \underline{M})^T \right) \in \underline{hom}^2(\underline{M}, \underline{M})\] is an identity for $\underline{M}$.
\end{remark}

\begin{remark}\label{rem:horiz_2comp_ham}(Horizontal composition of simple 2-morphisms) Let 
\begin{align*}\underline{M}, \underline{M}'&\in \underline{hom}^1(G,G'), \\
\underline{N}, \underline{N}'&\in \underline{hom}^1(G',G''), \\
L&\in shom^2(\underline{M}, \underline{M}'),\text{ and} \\
L' &\in shom^2(\underline{N}, \underline{N}').
\end{align*}
One could be tempted to define the horizontal compositon of $L$ and $L'$ as their cartesian product. 
Unfortunately it is generally not an element in $shom^2( \underline{M} \sharp^1_h \underline{N}  , \underline{M}' \sharp^1_h \underline{N}')$. For a counter-example, consider 
\[
(M,N)\colon G\to G'\to G'',
\]
the product of the diagonals $\Delta_M \times \Delta_N$ is not in the zero level of the moment map of the diagonal action of $G'$. A natural candidate for the identity 2-morphism of  $(M,N)$ would rather be the subset of $M^-\times M \times N^-\times N$ of elements $(m,m',n,n')$ such that $\mu_{G'}^{M} (m) +\mu_{G'}^{N} (n) = \mu_{G'}^{M} (m') +\mu_{G'}^{N} (n')$ and that $(m,n)$ and $(m',n')$ lie in the same orbit. This is an identity if the diagonal action of $G'$ is free, but is singular in general. For example if $G' = U(1)$, $M$ is a point, and $N=\cc$, acted on by $ U(1)$  by rotations with moment map $\mu(z) = \abs{z}^2$, this set consists in a single point.
\end{remark}

\begin{remark}\label{rem:Liec_Ham}(Relation between $\Liec$ and $\Ham$) Geometric invariant theory suggests a correspondence between (a variation of) $\Liec$ and $\Ham$. If $M$ is a Hamiltonian $G$-manifold and $J$ is a $G$-invariant almost complex structure, then under some conditions one can extend the action to an action by the complexification $G^\cc$, and the Kempf-Ness theorem says that the symplectic quotient $M\red G$ agrees with the GIT quotient of $M$ by $G^\cc$. This suggests a correspondence:
\begin{align*}
\Ham &\longleftrightarrow \Liec \\
 G &\longleftrightarrow G^\cc \\
 M&\longleftrightarrow  M\\
 L&\longleftrightarrow L,
\end{align*}
that should preserve the various operations. It would be interesting to define and study this correspondence in more detail.
\end{remark}

\begin{remark}\label{rem:end_trivial_group}
(Endomorphism category of the trivial group) The endomorphism category of the trivial group in $\Ham$, $End_{\Ham}(e)$, is similar with \WW's category $\Symp$, more precisely one has a functor from $\Symp$ to  $End_{\Ham}(e)$. However we should point out that this functor is not surjective, since there are 1-morphisms $e\to G_1 \to G_2 \to \cdots \to e$ that cannot be simplified to a length 1 sequence. For instance, the character variety of a closed surface $\Sigma$ is not an object of $\Symp$ due to its singular nature, however it can be represented by an object in $End_{\Ham}(e)$ using extended moduli spaces, as we shall see in Section~\ref{sec:moduli_spaces}. 
However, this functor is injective faithful, since if a sequence of simple 1-morphisms (resp. a diagram of simple 2-morphisms) in $\Ham$ can be simplified to a length 1 sequence (resp. a diagram with a single correspondence), then the resulting 1-morphism (resp. 2-morphism) is uniquely determined and given by the quotient of all intermediate groups.
\end{remark}

\subsection{Proof of the diagram axiom}
\label{ssec:diag_axiom}

We now prove that $\Lier$ and $\Ham$ both satisfy the diagram axiom, which implies that these are partial 2-categories, and allows one to define their strictifications.

The idea of proof is similar for both categories: we use the fact that even though compositions are only partially defined in these categories, these can always be defined in a set theoretical way: even if a group action is not free, one can always define the quotient set, likewise one can always define the reduction set of a Hamiltonian action. We will first prove the axiom at the set level, and then we will "lift" the statement to the initial framework, using the correspondence  between simple 2-morphisms and subsets of a (possibly singular) quotient set.

Recall the setting: let $\underline{\varphi}_0$, $\underline{\varphi}_1$, ... , $\underline{\varphi}_k = \underline{\varphi}_0$ be a sequence of representatives of general 1-morphisms in $\underline{hom}^1(x,y)$, such that for any $i$, $\underline{\varphi}_{i+1}$ is either a composition or a decomposition of $\underline{\varphi}_i$. To such a sequence is associated a diagram $\mathcal{D}$ by patching altogether all the identification 2-morphisms. We aim at proving that the diagram $\mathcal{D}$ is an identity for $\underline{\varphi}_0$.

For any $i$, let $\tilde{\varphi}_i$ stand for the set obtained by forcing all the compositions appearing in $\underline{\varphi}_i$:
\begin{itemize}
\item In $\Lier$, this is the quotient of the product of all the one morphisms appearing in $\underline{\varphi}_i$, modulo the diagonal actions of all the groups appearing in the sequence, except the first and the last.
\item In $\Ham$, this is the "symplectic quotient" of the product of all the one morphisms appearing in $\underline{\varphi}_i$, modulo the diagonal actions of all the groups appearing in the sequence, except the first and the last, namely the quotient set of the zero level of the associated moment map. 
\end{itemize}

Since we assumed that $\underline{\varphi}_{i+1}$ is either a composition or a decomposition of $\underline{\varphi}_i$, it follows that all the $\tilde{\varphi}_i$ are equal to the same set, which we will denote $\tilde{\varphi}$.

Let $ L_0\in shom^2(\underline{\psi}, \underline{\varphi}_0)$, for some $\underline{\psi}\in \underline{hom}^1(x,y)$. By composing it successively with all the identification morphisms (or their adjoint), we get for each $i$, a simple 2-morphism $L_i\in shom^2(\underline{\psi}, \underline{\varphi}_i)$, and we would like to prove that $ L_0 = L_k$.

For any $i$, Let $\tilde{L}_i$ be the subset of the  product of all the 1-morphisms appearing in $\underline{\psi}$ and $\tilde{\varphi}_i$, that corresponds to the quotient of $L_i$ (loosely speaking, $\tilde{L}_i \in ``shom^2(\underline{\psi}, \tilde{\varphi}_i)"$). By construction, and with the identification of the $\tilde{\varphi}_i$ in mind, all these correspond to the same subset, in particular $\tilde{L}_0 = \tilde{L}_k$.

Now, in either $\Lier$ or $\Ham$,  two simple morphisms in $shom^2(\underline{\psi}, \underline{\varphi}_0)$ are equal if and only if the corresponding subsets of $\left( \prod{\underline{\psi}} \right) \times \tilde{\varphi}_0 $ are equal. Indeed, an invariant subset is determined by its quotient in the orbit space. We therefore have $L_0 = L_k$. As $L_k = L_0 \circ \mathcal{D}$, we just proved that $\mathcal{D}$ is an identity for left composition. One can similarly show that for right composition. This completes the proof. \cqfd

\section{Moduli spaces of connections}
\label{sec:moduli_spaces}

Throughout this section we identify $SU(2)$ with the group of unit quaternions, and $SO(3)$ with its quotient by $\Z{2}$. Their common Lie algebra $\mathfrak{g} = \mathfrak{su(2)}$ then corresponds to the space of pure quaternions (with zero real part) and is equipped with the standard bi-invariant inner product of $\hh$, which we use to identify $\mathfrak{g}$ with its dual $\mathfrak{g}^*$. We will denote $B_{\mathfrak{g}}(\pi) \subset \mathfrak{g}$ the open ball of radius $\pi$, wich is sent injectively to $SU(2) \setminus \lbrace -1\rbrace$ by the exponential map.

\subsection{Moduli space of a closed 1-manifold}
\label{sec:moduli_space1}

Let ${C}$ be a circle with a base point. One can associate to it the moduli space ${G}({C})$ of framed $SU(2)$-{connections}, i.e. the quotient of the space of {connections} modulo gauge transformations that do not act on the base point. It is identified with the representation variety of its fundamental group, namely $Hom(\pi_1({C}), SU(2))$. Notice that since ${C}$ has an abelian fundamental group, a choice of a different basepoint yields a \emph{canonically} isomorphic space.

 If ${C}$ is oriented, its orientation furnishes a preferred generator of $\pi_1({C})$, which permits to identify  $G({C})$ with $SU(2)$ and therefore endows it with a group structure. If one reverse the orientation, the new generator corresponds to the inverse of the old one, and the identification with  $SU(2)$ differ by the inverse map. Therefore the group structure is changed to its opposite:
\[ 
G(- {C}) = G({C})^{op} .
\]
The group structure constructed can also be defined more topologically: given two elements of $G({C})$, with two {connections} $A_1$ and $A_2$ representing them, over bundles $P_1$ and $P_2$. One can cut $P_1$ and $P_2$ along the fiber over the base point of ${C}$, and glue them back as in Figure~\ref{fig:group_structure} to form a bundle over the gluig of the two segments, which one can identify with ${C}$, with the basepoint chosen as in the picture.

\begin{figure}[!h]
    \centering
    \def\svgwidth{.4\textwidth}
\begingroup%
  \makeatletter%
  \providecommand\color[2][]{%
    \errmessage{(Inkscape) Color is used for the text in Inkscape, but the package 'color.sty' is not loaded}%
    \renewcommand\color[2][]{}%
  }%
  \providecommand\transparent[1]{%
    \errmessage{(Inkscape) Transparency is used (non-zero) for the text in Inkscape, but the package 'transparent.sty' is not loaded}%
    \renewcommand\transparent[1]{}%
  }%
  \providecommand\rotatebox[2]{#2}%
  \newcommand*\fsize{\dimexpr\f@size pt\relax}%
  \newcommand*\lineheight[1]{\fontsize{\fsize}{#1\fsize}\selectfont}%
  \ifx\svgwidth\undefined%
    \setlength{\unitlength}{262.5bp}%
    \ifx\svgscale\undefined%
      \relax%
    \else%
      \setlength{\unitlength}{\unitlength * \real{\svgscale}}%
    \fi%
  \else%
    \setlength{\unitlength}{\svgwidth}%
  \fi%
  \global\let\svgwidth\undefined%
  \global\let\svgscale\undefined%
  \makeatother%
  \begin{picture}(1,1)%
    \lineheight{1}%
    \setlength\tabcolsep{0pt}%
    \put(0,0){\includegraphics[width=\unitlength]{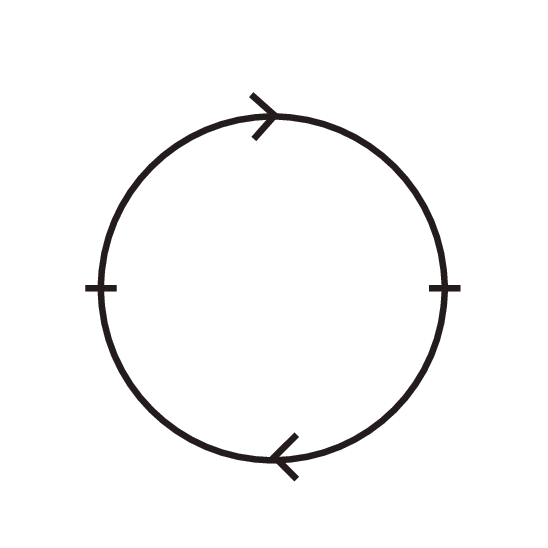}}%
    \put(0.25714286,1.79096888){\color[rgb]{0,0,0}\makebox(0,0)[lt]{\begin{minipage}{0.17142857\unitlength}\raggedright  \end{minipage}}}%
    \put(0.01697402,0.44428022){\color[rgb]{0,0,0}\makebox(0,0)[lt]{\lineheight{0}\smash{\begin{tabular}[t]{l}$*$\end{tabular}}}}%
    \put(0.39884007,0.85461242){\color[rgb]{0,0,0}\makebox(0,0)[lt]{\lineheight{0}\smash{\begin{tabular}[t]{l}$A_1$\end{tabular}}}}%
    \put(0.44174979,0.04818071){\color[rgb]{0,0,0}\makebox(0,0)[lt]{\lineheight{0}\smash{\begin{tabular}[t]{l}$A_2$\end{tabular}}}}%
  \end{picture}%
\endgroup%

      \caption{Group structure on $G({C})$.}
      \label{fig:group_structure}
\end{figure}
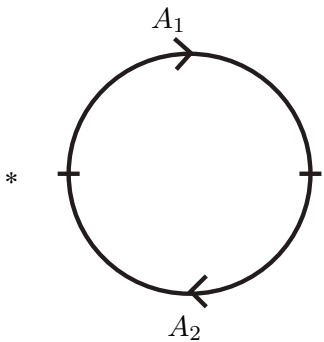

If now ${C} = {C}_1\sqcup \cdots \sqcup {C}_k $ is a disconnected closed oriented 1-manifold, pick a basepoint on each component, and define the group $G({C})$ as the product of the $G({C}_i)$'s.

\subsection{Moduli space of a surface with boundary}
\label{ssec:moduli_space2}

\begin{defi}\emph{(Extended moduli spaces)}\label{def:extmod}
Let $\Sigma$ be a compact oriented surface with boundary
\[
\partial \Sigma = \partial_1 \Sigma \sqcup \cdots \sqcup \partial_k \Sigma ,
\]
where the $\partial_i \Sigma$ stand for the connected components of $\partial \Sigma $.
\begin{itemize}
\item (Extended moduli space associated to a surface, \cite[Def. 2.1]{jeffrey}) 
Define the following space of flat {connections}:
\[ 
\mathscr{A}_F^\mathfrak{g}(\Sigma)  = \lbrace A \in \Omega^{1} (\Sigma )\otimes \mathfrak{su(2)}\ |\ F_A = 0,\  A_{|\nu \partial_i \Sigma} = \theta_i ds \rbrace, 
\]
where $\nu \partial \Sigma$ is a non-fixed tubular  neighborhood of $\partial \Sigma$, $s$  the parameter of $\rr/\zz$, and $\theta_i \in \mathfrak{g}$ is a constant element. The group
 \[ 
 \Gc (\Sigma ) = \left\lbrace u \colon \Sigma \rightarrow SU(2)\ |\ u_{|\nu \partial \Sigma} = 1 \right\rbrace 
 \]
acts by gauge transformations on $\mathscr{A}_F^\mathfrak{g}(\Sigma)$.
 
The extended moduli space is then defined as the quotient
 \[ 
  \Mg (\Sigma, p) = \mathscr{A}_F^\mathfrak{g}(\Sigma) /\Gc (\Sigma ),
  \]
This space carries a closed 2-form $\omega$ defined by:
\[ \omega_{[A]}([\alpha],[\beta]) = \int_{\Sigma} \langle \alpha\wedge\beta \rangle   ,\]
with $ [A]\in \Mg (\Sigma, p) $ and $\alpha,\beta$ representing tangent vectors at $[A]$ of $\Mg (\Sigma, p)$, namely $d_A$-closed $\mathfrak{su(2)}$-valued  1-forms, of the form $\eta_i ds$ near $\partial_i \Sigma$.

 Furthermore it has a Hamiltonian $SU(2)$-action, whose moment map is given by the elements $\theta_i \in \mathfrak{su(2)}$ such that $A_{|\nu \partial_i \Sigma} = \theta_i ds$. 

\item  Denote $\N(\Sigma, p)$ the subset of $\Mg(\Sigma, p)$ consisting in equivalence classes of connections for which $\abs{\theta_i} < \pi$. 
The form $\omega$ is symplectic on $\N(\Sigma, p)$ by \cite[Prop.~3.1]{jeffrey} (the proposition is stated and proved in the case when $\Sigma$ has connected boundary, but the same proof applies to any number of components).

\end{itemize}

\end{defi}

\begin{remark} The moduli space $\N(\Sigma,p)$ is noncompact. In \cite{MW}, \MW\ define Floer homology inside another moduli space $\Nc(\Sigma, p)$, which is a compactification of $\N(\Sigma,p)$ by symplectic cutting. We will ignore this compactification in this paper, as one can think of it just as a technical step in proving that Floer homology is well-defined inside $\N(\Sigma,p)$.
\end{remark}

\begin{remark}With the identification of $SU(2)$ with $G({C})$ in mind, one can think of the group action in a more topological way: Let $[A]$ be in $\N(\Sigma)$, and $g\in G({C})$, represented by a connection $A_g$ on the circle. Take a pair of pants ${P}$ with an embedded trivalent graph as in Figure~\ref{fig:group_action}. Let ${C}_1$, ${C}_2$ and ${C}_3$ denote its boundary components. Choose a trivial bundle $P$ over ${P}$, and equip it with a flat {connection} $A_P$  that corresponds to $A_{|\partial \Sigma}$ on ${C}_1$, and to $A_g$ on ${C}_2$. Gluing ${P}$ to $\Sigma$ and cutting along the trivalent graph yields a new connection on $\Sigma$ that corresponds to $g .[A]$.
\end{remark}

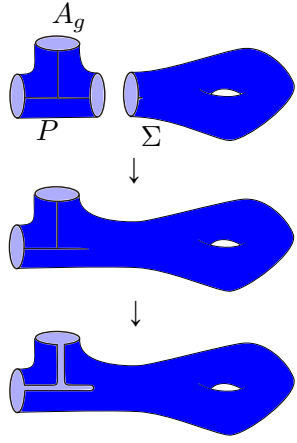
\begin{figure}[!h]
    \centering
    \def\svgwidth{.4\textwidth}
\begingroup%
  \makeatletter%
  \providecommand\color[2][]{%
    \errmessage{(Inkscape) Color is used for the text in Inkscape, but the package 'color.sty' is not loaded}%
    \renewcommand\color[2][]{}%
  }%
  \providecommand\transparent[1]{%
    \errmessage{(Inkscape) Transparency is used (non-zero) for the text in Inkscape, but the package 'transparent.sty' is not loaded}%
    \renewcommand\transparent[1]{}%
  }%
  \providecommand\rotatebox[2]{#2}%
  \newcommand*\fsize{\dimexpr\f@size pt\relax}%
  \newcommand*\lineheight[1]{\fontsize{\fsize}{#1\fsize}\selectfont}%
  \ifx\svgwidth\undefined%
    \setlength{\unitlength}{787.5bp}%
    \ifx\svgscale\undefined%
      \relax%
    \else%
      \setlength{\unitlength}{\unitlength * \real{\svgscale}}%
    \fi%
  \else%
    \setlength{\unitlength}{\svgwidth}%
  \fi%
  \global\let\svgwidth\undefined%
  \global\let\svgscale\undefined%
  \makeatother%
  \begin{picture}(1,1.23809524)%
    \lineheight{1}%
    \setlength\tabcolsep{0pt}%
    \put(0,0){\includegraphics[width=\unitlength]{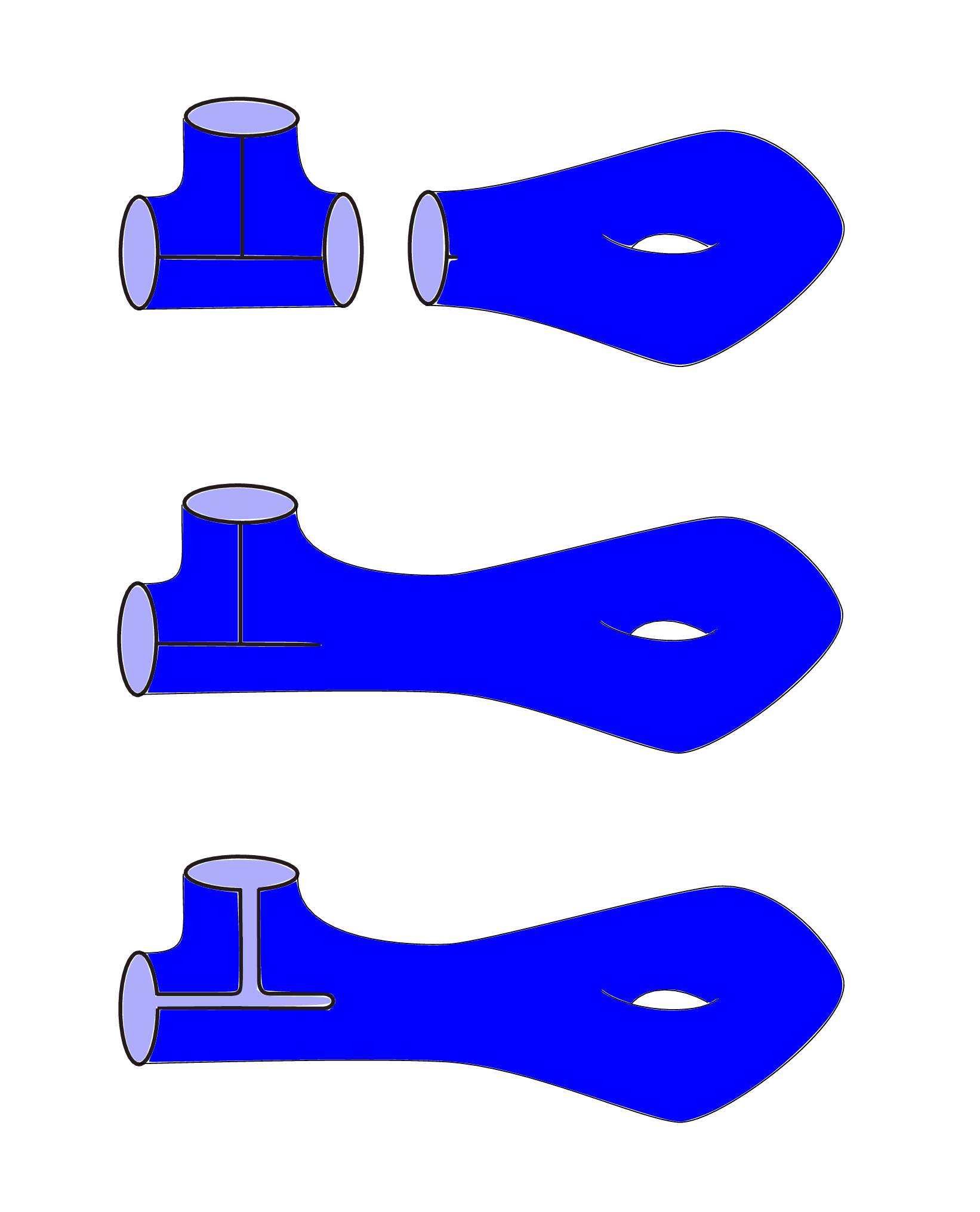}}%
    \put(0.08571429,0.59698963){\color[rgb]{0,0,0}\makebox(0,0)[lt]{\begin{minipage}{0.05714286\unitlength}\raggedright  \end{minipage}}}%
    \put(0.18986837,0.86535197){\color[rgb]{0,0,0}\makebox(0,0)[lt]{\lineheight{0}\smash{\begin{tabular}[t]{l}${P}$\end{tabular}}}}%
    \put(0.46358018,0.84952532){\color[rgb]{0,0,0}\makebox(0,0)[lt]{\lineheight{0}\smash{\begin{tabular}[t]{l}$\Sigma$\end{tabular}}}}%
    \put(0.23145271,1.17026653){\color[rgb]{0,0,0}\makebox(0,0)[lt]{\lineheight{0}\smash{\begin{tabular}[t]{l}$A_g$\end{tabular}}}}%
    \put(0.42649713,0.76457412){\color[rgb]{0,0,0}\makebox(0,0)[lt]{\lineheight{0}\smash{\begin{tabular}[t]{l}$\downarrow$\end{tabular}}}}%
    \put(0.43096372,0.39141953){\color[rgb]{0,0,0}\makebox(0,0)[lt]{\lineheight{0}\smash{\begin{tabular}[t]{l}$\downarrow$\end{tabular}}}}%
  \end{picture}%
\endgroup%

      \caption{Group action on $\N(\Sigma)$.}
      \label{fig:group_action}
\end{figure}

\begin{remark}\label{rem:explicit_description}(Explicit descrition, \cite[Sec.~6.2]{jeffrey})
Holonomies provide an explicit description of $\N(\Sigma,p)$.  Assume $\Sigma$ is connected. Each boundary component $\partial_i \Sigma$ has a basepoint $z_i$ corresponding to the image of $0\in \rr/\zz$ by the parametrizations. Let $\gamma_2$, ..., $\gamma_k$ be disjoin arcs connecting $z_1$ to $z_2$, ..., $z_k$, and let $\alpha_1$, $\beta_1$, ..., $\alpha_g$, $\beta_g$ be loops based at $z_1$ that form a symplectic basis of the fundamental group of $\Sigma \setminus \lbrace \gamma_2, {\ldots}, \gamma_k \rbrace$. With $A_i$, $B_i$ and $\Gamma_i$ denoting the holonomy of a {connection} along $\alpha_i$, $\beta_i$ and $\gamma_i$ respectively, and $\theta_i \in B_{\mathfrak{g}}(\pi)$ the value at $\partial_i \Sigma$, $\N(\Sigma,p)$ can be identified with the space of tuples
\[
 (\theta_1, {\ldots}, \theta_k,\Gamma_2, {\ldots}, \Gamma_k,A_1, B_1 {\ldots}, A_g, B_g ) \in (B_{\mathfrak{g}}(\pi))^{k} \times SU(2)^{k-1+2g}
\]
that satisfy the relation
\[
e^{\theta_1} \left( \Gamma_2 e^{\theta_2} \Gamma_2^{-1}\right)  \cdots \left( \Gamma_k e^{\theta_k} \Gamma_k^{-1}\right)  [A_1, B_1] \cdots [A_g, B_g] =1.
\]
And since $\theta_1 \in B_{\mathfrak{g}}(\pi)$, it is uniquely determined by the other elements and the above relation, so $\N(\Sigma,p)$ can be identified with the open (and hence smooth) subset of elements
\[
 (\theta_2, {\ldots}, \theta_k,\Gamma_2, {\ldots}, \Gamma_k,A_1, B_1 {\ldots}, A_g, B_g ) \in (B_{\mathfrak{g}}(\pi))^{k-1} \times SU(2)^{k-1+2g}
\] 
satisfying 
\[
 \left( \Gamma_2 e^{\theta_2} \Gamma_2^{-1}\right)  \cdots \left( \Gamma_k e^{\theta_k} \Gamma_k^{-1}\right)  [A_1, B_1] \cdots [A_g, B_g] \neq -1.
\]
\end{remark}
\begin{prop}\label{prop:gluing_equals_reduction}(Gluing almost equals reduction)
Let $\Sigma \in hom^1({C}_1,{C}_2)$ and $S \in hom^1({C}_2,{C}_3)$ be such that $\Sigma$, $S$ and $\Sigma\cup_{{C}_2} S$ have no closed components. Then one has a natural symplectomorphism
\[
 \N(\Sigma\cup_{{C}_2} S) \setminus C = \N(\Sigma) \redpr{G({C}_2)} \N(S)  ,
\]
where $C$ is the subset of {connections} whose holonomy around ${C}_2$ equals $-1$. Moreover, $C$ is a union of codimension 3 coisotropic submanifolds of $\N(\Sigma\cup_{{C}_2} S)$.
\end{prop}

\begin{proof}
If $A_\Sigma$ and $A_S$ denote two {connections} on $\Sigma$ and $S$ respectively that coincide on ${C}_2$, they can be glued together to a {connection} on $\Sigma\cup_{{C}_2} S$. This defines a gluing map 
\[
(\mu^{diag}_{G({C}_2)})^{-1}(0) \subset \N(\Sigma) \times \N(S) \to \N(\Sigma\cup_{{C}_2} S)
\]
that passes to the quotient for the diagonal $G({C}_2)$-action. One can see from the explicit description of the moduli spaces that the induced map is injective, and that its image is the complement of $C$, since the holonomies around ${C}_2$ correspond to the exponential of the $\theta_i$ values, which live in $B_{\mathfrak{g}}(\pi)$. Finally, this map preserves the symplectic forms, as both are defined in an analogous way, by integrating the forms on $\Sigma\sqcup S$ and $\Sigma\cup_{{C}_2} S$.
\end{proof}

\section{Construction of the functor}
\label{sec:constr}

We now define a $\Hamhat$-valued $(1+1+1)$-field theory. If a true ``gluing equals reduction'' principle would hold, i.e.  if there was no submanifold $C$ appearing in Proposition~\ref{prop:gluing_equals_reduction}, we would obtain a 2-functor from $\Cob_{1+1+1}$ to $\Hamhat$. Instead, we will obtain what we will call a quasi 2-functor: the source category will instead be $\Cob_{1+1+1}^{elem}$ (see Definition~\ref{def:cob_elem}), and consist {of} cobordisms equipped with decompositions. We then define an equivalence relation on morphism spaces of $\Hamhat$, and define a quasi-functor to be a functor from  $\Cob_{1+1+1}^{elem}$ such that a (2-)cobordism endowed with two different decompositions will result in two equivalent morphisms in $\Hamhat$.

\begin{remark} In future work, we expect to promote $\Hamhat$ to a (sort of) 3-category, using equivariant Floer homology as 3-morphism spaces. Such a construction should permit to define a linearization 2-functor $\mathcal{L}\colon \Hamhat \to \mathcal{C}$ that would land in a more algebraic 2-category (such as rings, $A_\infty$-algebras with ring actions, $A_\infty$-modules). We expect that the codimension 3 submanifolds $C$ of Proposition~\ref{prop:gluing_equals_reduction} should be invisible to equivariant Floer homology. 

This expectation, together with the observation in Remark~\ref{rem:recover_Lagrangians},  should imply that the composition of the functor with this linearization should descend to $\Cob_{1+1+1}$:
\[
\xymatrix{\Cob_{1+1+1}^{elem} \ar[d] \ar[r]^{} & \Hamhat \ar[d]^{\mathcal{L}} \\
\Cob_{1+1+1}  \ar@{-->}[r] & \mathcal{C} 
  } .
\]
Ultimately, one would also hope to extend such a 2-functor to a 3-functor.
\end{remark}

\subsection{Cobordisms and decompositions to elementary pieces}
\label{ssec:cob_and_decomp}

\begin{defi}\label{def:cob1+1+1}
Let $\Cob_{1+1+1}$ stand for the weak 2-category whose:

\begin{itemize}
\item Objects are oriented closed one-manifolds, endowed with an orientation preserving parametrization by $\rr/\zz$ of each connected component.

\item 1-morphisms  are compact oriented cobordisms (i.e. surfaces with boundary), 
\item 2-morphisms are diffeomorphism classes of compact oriented 3-manifolds with corners. That is, for $\Sigma$ and $S$ two 1-morphisms from $C_1$ to $C_2$, a 2-morphism from $\Sigma$ to $S$ is represented by a compact oriented 3-manifold with boundary $Y$, with a diffeomorphism between $\partial Y$ and $(-\Sigma)\cup_{C_1\cup C_2} S$. And two such $Y$ are identified if there is a diffeomorphism between them that is compatible with the identifications of the boundaries. 
\end{itemize}
All the compositions are given by glueing.
\end{defi}

We now define elementary morphisms. We use a slightly wider class than the usual notion of having a function with at most one critical point, as for example in \cite{WWfft}.

\begin{defi}(Elementary morphisms of $\Cob_{1+1+1}$).

\underline{Objects.} All objects are said to be elementary, including the empty set.

\underline{1-morphisms.} A 1-morphism is elementary if it has no closed components. This also includes the empty set.

Strictly speaking, elementary 2-morphisms will not be 2-morphisms between 1-morphisms of $\Cob_{1+1+1}$, but rather between sequences of elementary  1-morphisms, i.e. 1-morphisms with a given decomposition into elementary 1-morphisms.

\underline{2-morphisms.} Let $C_1$, $C_2$ be two objects, and $\Sigma$, $S$ be 1-morphisms from $C_1$ to $C_2$, endowed with decompositions $\underline{\Sigma} = (\Sigma_0, {\ldots}, \Sigma_k )$ and $\underline{S} =(S_0, {\ldots}, S_l)$ into elementary 1-morphisms (i.e. with no closed component).  An elementary 2-morphism $Y$ from  $\underline{\Sigma}$ to  $\underline{S}$ is  either:
\begin{itemize}
\item A \emph{compression body}:  $Y$ is a compression body if it admits a Morse function $f\colon (Y\setminus (C_1 \cup C_2 )) \to \rr$ that is minimal on $\Sigma$, maximal on $S$, vertical near $C_1$ and $C_2$ (meaning that, for an identification of a neighborhood of $C_i$ in $Y\setminus (C_1 \cup C_2$ with $C_i\times [0,1] \times (0,1]$, with $\Sigma$ and $S$ corresponding to $C_i\times \lbrace 0,1\rbrace \times (0,1]$, $f$ corresponds to the projection to $[0,1]$), and admits (or not) critical points that all have the same index, 1 or 2. Moreover, for some pseudo-grandient of $f$, the decomposition $\underline{\Sigma}$ of $\Sigma$ flows down to the decomposition $\underline{S}$ of $S$. This definition includes the empty set.

\item A 0-handle or a 3-handle attachment, i.e. $Y$ is a 3-ball.

\item A \emph{circle(s) insertion/removal.} We assume that $Y$ is trivial, i.e. it admits a Morse function $f\colon (Y\setminus (C_1 \cup C_2 )) \to \rr$ that is minimal on $\Sigma$, maximal on $S$, vertical near $C_1$ and $C_2$, and with no critical point. Furthermore we assume that there exists a pseudo-gradient for $f$ that matches the two decompositions, except for one circle (i.e.  $\underline{\Sigma}$ contains one circle more/less than $\underline{S}$).

\end{itemize}
\end{defi}

Elementary morphisms generate $\Cob_{1+1+1}$ in the following sense:

\begin{prop}\label{prop:elem_generate} (Existence  of decompositions)
\begin{enumerate}
\item Every 1-morphism in $\Cob_{1+1+1}$ decomposes as a sequence of elementary ones. Moreover, one can find such a sequence of length 2.
\item  Every 2-morphism in $\Cob_{1+1+1}$ decomposes as a sequence of elementary 2-morphism between sequences of elementary 1-morphisms, Moreover such a sequence can also be chosen to have length 2.
\end{enumerate}
\end{prop}

\begin{proof}
\begin{enumerate}
\item If a 1-morphism $\Sigma$ has closed components, one can just insert a separating curve in each of them, this splits the surface into two elementary 1-morphisms.
\item This follows from the fact that one can find self-indexed Morse functions
\end{enumerate}
\end{proof}

It follows from Cerf theory that decompositions of morphisms of $\Cob_{1+1+1}$ are unique in the following sense:

\begin{prop}\label{prop:Cerf_moves}(Uniqueness of decompositions)
Two decompositions of a morphism of $\Cob_{1+1+1}$ into elementary morphisms can be related by a sequence of moves of the following type.

\underline{For 1-morphisms:} 
\begin{itemize}
\item (Circle(s) insertion/removal) Replacing a length $k$ sequence 
\[
\underline{\Sigma} = (\Sigma_1, {\ldots} ,\Sigma_i, {\ldots} , \Sigma_k)
\] 
by a length $k+1$ one $(\Sigma_1, {\ldots} ,\Sigma_i^1,\Sigma_i^2, {\ldots} , \Sigma_k)$, where $(\Sigma_i^1,\Sigma_i^2)$ are obtained from $\Sigma_i$ by inserting some separating circles. Circle removal is the opposite move.
\end{itemize}

\underline{For 2-morphisms:} Let $\underline{Y} = (Y_0, {\ldots} , Y_k)$ be a sequence of elementary 2-morphisms, with $Y_i$ from $\underline{\Sigma}_i$ to $\underline{\Sigma}_{i+1}$.
\begin{itemize}
\item (Diffeomorphism equivalence) Let $\underline{Y}' = (Y_0', {\ldots} , Y_k')$ be a sequence of elementary 2-morphisms, with $Y_i'$ from $\underline{\Sigma}_i'$ to $\underline{\Sigma}_{i+1}'$, with a family of {diffeomorphisms} $\varphi_i \colon Y_i \to Y_i'$ such that $\varphi_i$ and $\varphi_{i+1}$ coincide on $\underline{\Sigma}_{i+1}$ and send the decomposition $\underline{\Sigma}_{i+1}$ to $\underline{\Sigma}_{i+1}'$. A diffeomorphism equivalence consists in replacing $\underline{Y}$ by $\underline{Y}'$.

\item (Cylinder creation/cancellation)
 Let $\underline{\Sigma} = (\Sigma_1, {\ldots} , \Sigma_k)$ be a sequence of elementary 1-morphisms, with total space $\Sigma = \Sigma_1 \cup \cdots \cup  \Sigma_k$. By a cylinder from $\underline{\Sigma}$ to itself we mean a compression body with no critical points. A cylinder creation corresponds to inserting a cylinder at $\underline{\Sigma}_i$, for some $i$. A cylinder cancellation is the opposite move.

\item (Circle(s) insertion/removal) A circle insertion corresponds to keeping the same 3-manifolds $Y_i$, and performing a circle insertion on one of the $\underline{\Sigma_i}$'s. A circle removal  is the opposite move.

\item (Imbrication of compression bodies) Assume that for some $i$, $Y_i$ and $Y_{i+1}$ are compression bodies of the same index, then their union $Y_i\cup_{\Sigma_{i+1}}Y_{i+1}$ is again a compression body. The move is to replace $\underline{Y}$ by 
\[
(Y_0, {\ldots} ,Y_i\cup_{\Sigma_{i+1}}Y_{i+1} , {\ldots} , Y_k).
\]

\item (Critical point switches) Assume that for some $i$, $Y_i$ and $Y_{i+1}$ correspond to handle attachments along disjoint attaching spheres. The move is to replace $\underline{Y}$ by 
\[
(Y_0, {\ldots} ,\widetilde{Y}_i, \widetilde{Y}_{i+1} , {\ldots} , Y_k),
\]
where $\widetilde{Y}_i$ corresponds to attaching the handles of $Y_{i+1}$ first, and $\widetilde{Y}_{i+1}$ corresponds to attaching the handles of $Y_{i}$ afterwards.

\item (Index 0-1 (or 2-3) handle creation/cancellation). Assume that for some $i$, $Y_i$ is a $0$-handle attachment, with
\begin{align*}
\underline{\Sigma}_i &= (\Sigma_i^1,\Sigma_i^2, {\ldots}, \Sigma_i^k )\text{, and} \\
\underline{\Sigma}_{i+1} &= (D^0, D^1, \Sigma_i^1,\Sigma_i^2, {\ldots}, \Sigma_i^k )\text{,} \\
\end{align*}
where $D^0$ and $D^1$ are two 1-discs such that their union is the 2-sphere bounding the 0-handle. Assume also that $Y_{i+1}$ corresponds to a 1-handle attachment connecting $D^1$ and $\Sigma_i^1$, so that
\[
\underline{\Sigma}_{i+2} = (D^0, D^1 \sharp \Sigma_i^1,\Sigma_i^2, {\ldots}, \Sigma_i^k ).
\]
Assume finally that $Y_{i+2}$ corresponds to the removal of the circle between $D^0$ and $D^1$, so that
\[
\underline{\Sigma}_{i+3} = (D^0\cup ( D^1 \sharp \Sigma_i^1),\Sigma_i^2, {\ldots}, \Sigma_i^k ).
\]
The cancellation move is to replace $(Y_{i},Y_{i+1},Y_{i+2})$ by the cylinder $(Y_{i} \cup Y_{i+1} \cup Y_{i+2})$. The creation move corresponds to the opposite move. The similar moves for 2-3 handles corresponds to the reversed cobordisms $(\overline{Y_{i+2}},\overline{Y_{i+1}},\overline{Y_{i}})$.

\item (Index 1 and 2 handle creation/cancellation) If now for some $i$, $Y_i$ and $Y_{i+1}$ are either compression bodies of index 1 and 2 respectively, and such that their union $Y_i\cup_{\Sigma_{i+1}}Y_{i+1}$ is a cylinder. The move is to replace $\underline{Y}$ by 
\[
(Y_0, {\ldots} ,Y_i\cup_{\Sigma_{i+1}}Y_{i+1} , {\ldots} , Y_k).
\]

\end{itemize}
\end{prop}

\begin{proof}
1-morphisms: Let $\underline{\Sigma}$ and $\underline{\Sigma}'$ be two elementary decompositions of a given 1-morphism $\Sigma$. On each closed component of $\Sigma$, pick any separating circle disjoint from the {circles} of the decompositions $\underline{\Sigma}$ and $\underline{\Sigma}'$: call $\underline{C}$ the collection of these circles. One can go from $\underline{\Sigma}$ to $\underline{\Sigma}'$ by 
\begin{itemize}
\item first adding the circles $\underline{C}$,
\item removing all the circles corresponding to  $\underline{\Sigma}$,
\item adding all the circles corresponding to  $\underline{\Sigma}'$,
\item removing the circles $\underline{C}$.
\end{itemize}

2-morphisms: The statement about the pieces $Y_i$, ignoring the decompositions of the level surfaces $\Sigma_i$, is standard Cerf theory, see for example \cite[Th.~2.2.11]{WWfft}. 
To get the correct decompositions $\underline{\Sigma_i}$, one can then insert cylinders and run the same method as above for 1-morphisms.

\end{proof}

We now define a 2-category $\Cob_{1+1+1}^{elem}$, where we keep track of decompositions:
\begin{defi}\label{def:cob_elem} ($\Cob_{1+1+1}^{elem}$) Let $\Cob_{1+1+1}^{elem}$ stand for the 2-category whose:
\begin{itemize}
\item Objects are the same as in $\Cob_{1+1+1}$, 
\item One-morphisms consist in sequences 
\[ 
\underline{\Sigma}= (\Sigma_0, {\ldots} , \Sigma_k)
\]
of elementary 1-morphisms (or equivalently 1-morphism of $\Cob_{1+1+1}$ endowed with a decomposition into elementary morphisms).

\item 2-morphisms from $\underline{\Sigma}$ to $\underline{\Sigma}'$  consist in sequences
\[ 
\underline{Y}= (Y_0, {\ldots} , Y_k), 
\]
where $Y_i$ is an elementary 2-morphism from $\underline{\Sigma}_i$ to $\underline{\Sigma}_{i+1}$, with $\underline{\Sigma}_0 = \underline{\Sigma} $ and $\underline{\Sigma}_{k+1} = \underline{\Sigma}'$.

\item Compositions are given by concatenation.

\end{itemize}
\end{defi}

\subsection{An equivalence relation on $\Hamhat$}
\label{ssec:equiv_rel_hamhat}

Recall that the moduli space associated to the gluing of two surfaces $\Sigma$ and $S$ does not exactly correspond to the composition $\N(\Sigma) \redpr{} \N(S)$, due to the presence of the subset $C$ in Proposition~\ref{prop:gluing_equals_reduction}. Therefore  the construction of Section~\ref{sec:moduli_spaces} does not define a functor with values in $\widehat{\Ham}$, but rather a ``quasi 2-functor'' modulo the  equivalence relation of Definition~\ref{def:equiv_rel_hamhat} below. We will mod out the 1-morphism spaces  by identifying Hamiltonian manifolds up to codimension 3  submanifolds. Care must be taken however at the level of 2-morphisms, since the Lagrangian correspondences we will be considering can always be contained in such codimension 3  submanifolds. In order to define a nontrivial functor, one has to take this fact into account, this is the reason for which we introduce the weak transversality assumption.

\begin{defi}\label{def:weakly_transverse}
Let $A$ and $B$ be two subsets of a topological space $X$. We say that $A$ intersects $B$ in a \emph{weakly transverse} way if $A = \overline{(A\setminus B)}$.
\end{defi}

\begin{defi}\label{def:equiv_rel_hamhat}
Define the following equivalence relation on $\Hamhat$:

\begin{itemize}
\item Objects are equivalent if and only if they are equal.
\item The equivalence relation on $hom^1(G,G')$ is generated by the following identifications: let $\underline{M} = (M_0, {\ldots}, M_k)\in \underline{hom}^1(G,G')$ be a representative of a 1-morphism, and let $C\subset M_i$ be a codimension 3 submanifold. Then we identify $\underline{M}$ with 
\[
(M_0, {\ldots},  M_i \setminus C ,{\ldots}, M_k).
\]

\item On $hom^2(\alpha, \beta)$: suppose that for some representatives $\underline{M}$ and $\underline{N}$ of $\alpha$ and $\beta$ respectively, we have a diagram $\mathcal{D}\in \underline{hom}^2(\underline{M},\underline{N})$. Let $M$ be some 1-morphism decorating an edge in $\mathcal{D}$, and $C\subset M$ be a (coisotropic) codimension 3 submanifold such that for any face in $\mathcal{D}$, decorated by a 2-morphism $L$, which by cyclicity we can think of as an element in $shom^2(\underline{P}, M)$. Suppose then that inside $(\prod{\underline{P}} )\times M$, the intersection of $L$ with $(\prod{\underline{P}}) \times C$ is  weakly transverse, in the sense of Definition~\ref{def:weakly_transverse}. If that transversality assumption holds for any $L$ adjacent to $M$, then we declare to be equivalent the 2-morphisms of $hom^2(\alpha, \beta)$ associated with $\mathcal{D}$ and $\mathcal{D}'$, the new diagram built from $\mathcal{D}$ by replacing $M$ by $M\setminus C$, any adjacent $L$ by $L\setminus \left( (\prod{\underline{P}}) \times C \right) $, and leaving the rest of the diagram unchanged.

\end{itemize}
\end{defi}

\begin{remark} \label{rem:recover_Lagrangians} 
In the setting described above, if one is given the diagram $\mathcal{D}'$ and $M$, then one can recover the diagram $\mathcal{D}$: any adjacent $L$ corresponds to the closure of $L\setminus \left( (\prod{\underline{P}}) \times C \right) $ inside $(\prod{\underline{P}}) \times M$.
\end{remark}

\subsection{Construction on elementary morphisms}
\label{ssec:functor_elem_morphisms}

We now define the quasi 2-functor 
\[
\Cob^{elem}_{1+1+1} \to \Hamhat.
\]

To a 0-morphism $C$, we associate the group $G(C)$ defined in Section~\ref{sec:moduli_space1}.

To an elementary 1-morphism $\Sigma$, we associate the moduli space  $\N(\Sigma)$ defined in Section~\ref{ssec:moduli_space2} (4.2).

If now $\underline{\Sigma} = (\Sigma_1, {\ldots} , \Sigma_k)$ is a sequence of elementary 1-morphisms
\[
\underline{\Sigma} = \xymatrix{ C_0 \ar[r]^{\Sigma_1} & C_1 \ar[r]^{\Sigma_2} & \cdots \ar[r]^{\Sigma_k}& C_k   },
\]
we associate to it the corresponding sequence $\underline{M} = (\N(\Sigma_1), {\ldots} , \N(\Sigma_k))$. 
 
\begin{lemma}\label{lem:indep_param_1mph} As a 1-morphism of $\Hamhat$, the sequence $(\N(\Sigma_1), {\ldots} , \N(\Sigma_k))$ is independent on the choice of parametrizations of the intermediate 1-manifolds $C_1$, ..., $C_{k-1}$. 
\end{lemma} 

\begin{proof}

Let $p_i, p_i '\colon \rr/\zz \sqcup \cdots \sqcup \rr/\zz \to C_i$ be two parametrizations of $C_i$, with $1\leq i\leq k-1$. Insert a cylinder $ [0,1] \times C_i$ between $\Sigma_i$ and $\Sigma_{i+1}$, and parametrize $\lbrace 0 \rbrace \times C_i$ and $\lbrace 1 \rbrace \times C_i$ by $p_i$ and $p_i'$ respectively, so to have a sequence
\[
\xymatrix{\cdots \ar[r]^{\Sigma_i} & (C_i,p_i) \ar[rr]^{C_i\times [0,1]} & & (C_i,p_i')  \ar[r]^{\Sigma_{i+1}} & \cdots
}
\]
and an associated sequence in $\Hamhat$:
\[
\xymatrix{\cdots \ar[r]^{\N(\Sigma_i,p_i)} & G(C_i,p_i) \ar[rr]^{N(C_i\times [0,1], p_i,p_i ')} &  & G(C_i,p_i ')  \ar[r]^{\N(\Sigma_{i+1},p_i ')} & \cdots
}
\]
composing respectively at $G(C,p_0)$ and $G(C,p_1)$, we obtain that the two sequences
\[
\xymatrix{\cdots \ar[r]^{\N(\Sigma_i,p_i)} & G(C_i,p_i)  \ar[r]^{\N(\Sigma_{i+1},p_i)} & \cdots \\
\cdots \ar[r]^{\N(\Sigma_i,p_i ')} & G(C_i,p_i ')  \ar[r]^{\N(\Sigma_{i+1},p_i ')} & \cdots 
}
\]
define the same morphism of $\Hamhat$.

\end{proof}

To an elementary 2-morphism we will associate a diagram of Lagrangian correspondences. Assume first that $Y$ is a compact 3-manifold with boundary $\partial^1 Y$ and codimension 2 corner 
\[
\partial^2 Y = C_1 \sqcup \cdots \sqcup C_k \subset \partial^1 Y,
\]
 and let $\Sigma$ be the compact surface with boundary obtained by cutting $\partial^1 Y$ along $\partial^2 Y$. We parametrize each circle $C_i$ in $\partial^2 Y$, and take the induced parametrization of $\partial \Sigma$.

\begin{defi}\label{def:lagcorr_3mfds} Let $Y$ and $\Sigma$ be as above, 
\begin{itemize}
\item (Moduli space associated to $Y$) Let $\N(Y) = \A_F (Y) / \Gc(Y)$, where
\[
\A_F (Y) = \left\lbrace A\in \Omega^1(Y)\otimes \g  \ |\ F_A=0, A_{|\nu C_i} = \theta_i ds \right\rbrace,
\]
with $\theta_i \in \Bg$, $\nu C_i$ a non-fixed neighborhood of $C_i$, and $s\in \rr/\zz$ the parameter of $C_i$; and
\[
\Gc (Y) = \left\lbrace u\colon Y\to SU(2)  \ |\ u_{|\nu C_i} = 1 \right\rbrace.
\]
\item (Correspondence associated to $Y$) Let $L(Y)\subset \N(\Sigma)$ denote the image of $\N(Y)$ by the restriction map to $\Sigma$.
\end{itemize}
\end{defi}

For arbitrary $Y$, $L(Y)$ might not be smooth. However, when this is the case, Stokes formula implies that the map $\N(Y) \to \N(\Sigma)$ is Lagrangian. Indeed, tangent vectors at $[A]$ of $\N(Y)$ can be represented by classes of $\g$-valued 1-forms that are $d_A$-closed, and locally constant near $\partial^2 Y$. For two such forms $\alpha, \beta$,
\[
\omega([\alpha], [\beta]) = \int_{\Sigma}{\left\langle \alpha \wedge \beta\right\rangle } =  \int_{Y}{d\left\langle \alpha \wedge \beta\right\rangle } =0.
\]
We will see that $L(Y)$ is a smooth embedded Lagrangian for 0-handle attachments, and for \emph{some} compression bodies and circle insertions.

\vspace{.3cm}
\paragraph{\underline{0-handle attachments}} In this case, $Y$ is a 3-ball, and $\Sigma$ consists in a disjoint union of punctured spheres $\Sigma_i$. For each component, by the explicit description in Remark~\ref{rem:explicit_description}, $\N(\Sigma_i)$ is identified with an open subset of $\Bg^{k_i -1} \times SU(2)^{k_i -1}$, where $k_i$ stands for the number of boundary components. Since all the boundary circles bound discs in $Y$, under the previous identification, $L(Y)$ corresponds to the subset
\[
\prod_i{ \left\lbrace 0 \right\rbrace  \times SU(2)^{k_i -1}} \subset \prod_i{\Bg^{k_i -1} \times SU(2)^{k_i -1}},
\]
as the holonomy along the $\gamma$-curves joining the boundaries can take arbitrary values. It follows that $L(Y)$ is smooth, and therefore Lagrangian.

\vspace{.3cm}
\paragraph{\underline{Circle insertion}} Assume now that $Y\colon \underline{\Sigma} \to \underline{S}$ corresponds to a circle insertion as defined in Definition~\ref{def:cob_elem}, and assume moreover that $\underline{\Sigma}$ has length two, and $\underline{S}$ length one, i.e.
\begin{align*}
\underline{\Sigma} &= \xymatrix{C_0 \ar[r]^{\Sigma_1} & C_1 \ar[r]^{\Sigma_2} & C_2  }\text{, and} \\
\underline{S} &= \xymatrix{C_0 \ar[r]^{S} & C_2  },
\end{align*}
with $S\simeq \Sigma_1\cup_{C_1}\Sigma_2$. From Proposition~\ref{prop:gluing_equals_reduction}, we have the identification
\[
\N(S)\setminus C = \N(\Sigma_1) \redpr{G(C_1)} \N(\Sigma_2),
\]
 under which $L(Y)$ corresponds to the identification 2-morphism 
 \[
 I_{G(C_1)}(\N(\Sigma_1), \N(\Sigma_2)),
 \]
  modulo the codimension 3 subset $C$. Indeed, if $A_1$ and $A_2$ are flat {connections} on $\Sigma_1$ and $\Sigma_2$  that extend flatly to $Y$, then they must coincide on $C_1$, which corresponds to the condition of being in the zero level of the moment map for the diagonal $G(C_1)$-action. And since $A$ is flat and $Y\setminus \left( C_0 \cup C_2 \right)$ is a trivial cobordims, the restriction $A_{|S}$ must be gauge equivalent to the gluing of $A_1$ and $A_2$ on $\Sigma_1\cup_{C_1}\Sigma_2$. It follows from the explicit description of the moduli spaces that $0$ is a regular value of the moment maps (and therefore the diagonal moment map), and that the $G(C_1)$-action is free. This imply that $I_{G(C_1)}(\N(\Sigma_1), \N(\Sigma_2))$, and therefore $L(Y)$, are smooth.

\vspace{.3cm}
\paragraph{\underline{2-handle attachmnents}} Assume that $Y\colon \Sigma\to S$ corresponds to 2-handle attachments on $\Sigma$ which we assume to be connected (but $S$ may be disconnected, yet with no closed components). We will show that $L(Y)$ is induced by a fibered coisotropic submanifold on $\N(\Sigma)$, in the following sense:

\begin{defi}\label{def:fibered_coisotropic} A coisotropic submanifold $C\subset M$ of a symplectic manifold $(M,\omega)$ is said to be \emph{fibered} if its characteristic foliation $TC^{\bot_\omega} \subset TC$ corresponds to the vertical foliation $\ker dp$ of a fibration $p\colon C\to B$. In this case, $\omega$ induces a symplectic form on $B$, and $L = (i\times p)(C) \subset M^-\times B$ is a Lagrangian correspondence.
We will say that a Lagrangian correspondence $\Lambda\subset M^-\times M'$ is \emph{induced} by $C$ if for some symplectomorphism  $B\simeq M'$, $\Lambda$ corresponds to $L$.
\end{defi}

From the following proposition, it is enough to consider the case when $Y$ is a single handle attachment.

\begin{prop}\label{prop:compo_fibered_coisotropics} Let $M_0$, $M_1$, $M_2$ be three symplectic manifolds, and $L_{01}\subset M_0^-\times M_1$, $L_{12}\subset M_1^-\times M_2$ be Lagrangian correspondences induced by fibered coisotropics $C_{01} \subset M_0$, and $C_{12} \subset M_1$. Then the composition of $L_{01}$ with $L_{12}$ is embedded, and is induced by the fibered coisotropic $C_{02} = p_{01}^{-1}(C_{12})$, where $p_{01} \colon C_{01}\to M_1$ denotes the fibration.
\end{prop}

\begin{proof}
Let $(x_0, x_1, x_2)\in L_{01}\times M_1 \cap M_0 \times L_{12}$, we have $x_0\in C_{02}$, $x_1 = p_{01}(x_0)$, and  $x_2 = p_{02}(x_0)$, with 
\[
p_{02} = p_{12} \circ p_{01} \colon C_{02} \to M_2 .
\]
 It follow by differentiating that 
\begin{align*}
&\left( D_{x_0}{(i_{01}\times p_{01})}(T_{x_0}C_{01}) \times T_{x_2}M_2\right)  \cap \left( T_{x_0}M_0 \times D_{x_1}{(i_{12}\times p_{12})}(T_{x_1}C_{12})\right)  \\
&=D_{x_0}(i_{02}\times p_{01} \times p_{02})(  T_{x_0}C_{02}),
\end{align*}
and $L_{01} \circ L_{12} = (i_{02}\times p_{02})(C_{02})$.
\end{proof}

Assume now that $\Sigma$ is connected, and $Y\colon \Sigma \Rightarrow S$ corresponds to a single 2-handle attachment. Assume first that the attaching circle is separating in $\Sigma$, and cuts it in two surfaces $\Sigma_1$ and $\Sigma_2$. Denote $C_1^1$, ..., $C_{k_1}^1$ and $C_1^2$, ..., $C_{k_2}^2$ the components of $\partial \Sigma$ contained respectively in $\Sigma_1$ and $\Sigma_2$. By assumption, $k_1, k_2 \geq 1$, otherwise $S$ would have a closed component. Each of these circles have a basepoint, corresponding to the image of $[0]\in\rr/\zz$. Pick disjoint embedded paths $\gamma_2^1$, ..., $\gamma_{k_1}^1$ and $\gamma_2^2$, ..., $\gamma_{k_2}^2$ joining the basepoint of $C_1^1$ (resp. $C_1^2$) to the other boundary components. Connect also $C_1^1$ and $C_1^2$ by a path $\delta$ disjoint from the other curves, and meeting the attaching circle at one point. Fix finally $\alpha_1^i$, $\beta_1^i$, ..., $\alpha_{g_i}^i$, $\beta_{g_i}^i$ a symplectic basis of the fundamental group of $\Sigma_i \setminus \left(\delta \cup\gamma_2^1 \cup \cdots \cup \gamma_{k_1}^1  \right)$ based at (the basepoint of) $C_1^i$. See Figure~\ref{fig:separating_curve}. These curves flow down to analogous curves on $S$.

\begin{figure}[!h]
    \centering
    \def\svgwidth{.6\textwidth}
\begingroup%
  \makeatletter%
  \providecommand\color[2][]{%
    \errmessage{(Inkscape) Color is used for the text in Inkscape, but the package 'color.sty' is not loaded}%
    \renewcommand\color[2][]{}%
  }%
  \providecommand\transparent[1]{%
    \errmessage{(Inkscape) Transparency is used (non-zero) for the text in Inkscape, but the package 'transparent.sty' is not loaded}%
    \renewcommand\transparent[1]{}%
  }%
  \providecommand\rotatebox[2]{#2}%
  \newcommand*\fsize{\dimexpr\f@size pt\relax}%
  \newcommand*\lineheight[1]{\fontsize{\fsize}{#1\fsize}\selectfont}%
  \ifx\svgwidth\undefined%
    \setlength{\unitlength}{759bp}%
    \ifx\svgscale\undefined%
      \relax%
    \else%
      \setlength{\unitlength}{\unitlength * \real{\svgscale}}%
    \fi%
  \else%
    \setlength{\unitlength}{\svgwidth}%
  \fi%
  \global\let\svgwidth\undefined%
  \global\let\svgscale\undefined%
  \makeatother%
  \begin{picture}(1,0.49407115)%
    \lineheight{1}%
    \setlength\tabcolsep{0pt}%
    \put(0,0){\includegraphics[width=\unitlength]{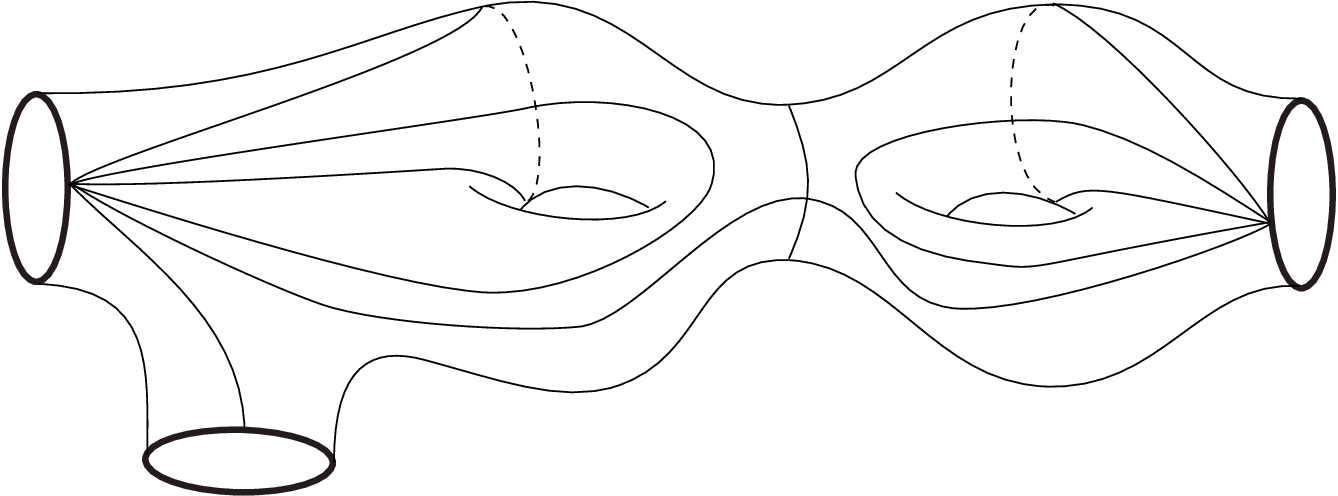}}%
    \put(0.30181969,0.40299134){\color[rgb]{0,0,0}\makebox(0,0)[lt]{\lineheight{1.25}\smash{\begin{tabular}[t]{l}$\Sigma_1$\end{tabular}}}}%
    \put(0.70228179,0.40612832){\color[rgb]{0,0,0}\makebox(0,0)[lt]{\lineheight{1.25}\smash{\begin{tabular}[t]{l}$\Sigma_2$\end{tabular}}}}%
  \end{picture}%
\endgroup%

      \caption{Handle attachment along a separating curve.}
      \label{fig:separating_curve}
\end{figure}

Let $\theta^i_j\in \Bg$ denoting the values of a {connection} at $C_j^i$, and $A_j^i$, $B_j^i$, $\Gamma_j^i$ the holonomies along the corresponding curves, $\Delta$ the holonomy along $\delta$, and let
\begin{align*}
 \Pi_1  &= \prod_{i=1}^{g_1}{[A_i^1,  B_i^1]} \prod_{i=2}^{k_1}{ad_{\Gamma_i^1} } e^{\theta_i^1}, \\
 \Pi_2  &= \prod_{i=1}^{g_2}{[A_i^2,  B_i^2]} \prod_{i=2}^{k_2}{ad_{\Gamma_i^2} } e^{\theta_i^2},
\end{align*}
corresponding to the holonomies along a loop going aroung the attaching circle.  The spaces $\N(\Sigma)$ and $\N(S)$ admit the following description:
\begin{align*}
\N(\Sigma) &= \left\lbrace \left( A_1^1,  B_1^1, {\ldots}, A_{g_1}^1,  B_{g_1}^1, \Gamma_2^1, {\ldots} , \Gamma_{k_1}^1,  \theta_2^1, {\ldots} , \theta_{k_1}^1, \right. \right. \\
& \left. A_1^2,  B_1^2, {\ldots}, A_{g_1}^2,  B_{g_1}^2, \Gamma_2^2, {\ldots} , \Gamma_{k_1}^2,  \theta_1^2, \theta_2^2, {\ldots} , \theta_{k_1}^2, \Delta \right)\ | \\
&\left. \Pi_1 \Delta \Pi_2 e^{\theta_1^2}  \Delta^{-1} \neq -1 \right\rbrace, 
\end{align*} 

\begin{align*}
\N(S) &= \left\lbrace \left( A_1^1,  B_1^1, {\ldots}, A_{g_1}^1,  B_{g_1}^1, \Gamma_2^1, {\ldots} , \Gamma_{k_1}^1,  \theta_2^1, {\ldots} , \theta_{k_1}^1, \right. \right. \\
& \left. A_1^2,  B_1^2, {\ldots}, A_{g_1}^2,  B_{g_1}^2, \Gamma_2^2, {\ldots} , \Gamma_{k_1}^2,   \theta_2^2, {\ldots} , \theta_{k_1}^2 \right)\ | \\
&\left. \Pi_1 \neq -1,\  \Pi_2  \neq -1 \right\rbrace, 
\end{align*} 
(we have dropped $\theta_1^2$ and $\Delta$ in $\N(S)$).  With $C = \left\lbrace \Pi_1=1 \right\rbrace \subset \N(\Sigma)$, $C$ is fibered over $\N(S)$, where $p\colon C\to \N(S)$ is given by forgetting $\Delta$ and $\theta_1^2$ (indeed, $C\simeq SU(2) \times \N(S)$, where $\Delta$  corresponds to the $SU(2)$-factor), and $L(Y)$ is the correspondence induced by $C$ (which is coisotropic, since $L(Y)$ is Lagrangian).

The case when the attaching circle is nonseparating is similar, and can be described in an analogous way by assuming that the attaching circle is one of the $\alpha$-{curves}, see \cite[Sec.~5.2.1]{surgery}.

If now $Y\colon \Sigma \Rightarrow S$ is an arbitrary handle attachment, one can decompose it in single handle attachments. By Proposition~\ref{prop:compo_fibered_coisotropics} and the above discussion, $L(Y)$ is again induced by a fibered coisotropic, which corresponds to the set of {connections} having trivial holonomies along the attaching circles.

\begin{remark} If $Y$ is a more general circle insertion or handle attachment, i.e. the sequences $\underline{\Sigma}$ and $\underline{S}$ can be longer, we have not checked whether or not $L(Y)$ is smooth, but a priori there can be issues such as in Remark~\ref{rem:horiz_2comp_ham}. This is the reason why in these cases we associate more complicated diagrams, containing only correspondences as above.
\end{remark}

\vspace{.3cm}
\paragraph{\underline{General circle insertions}}

Assume that $Y\colon \underline{\Sigma} \Rightarrow \underline{S}$, with 
\begin{align*}
\underline{\Sigma} &= (\Sigma_1, {\ldots}, \Sigma_i, \Sigma_{i+1}, {\ldots}, \Sigma_k ) \text{, and} \\
\underline{S} &= (\Sigma_1, {\ldots}, \Sigma_i \cup \Sigma_{i+1}, {\ldots}, \Sigma_k ) ,
\end{align*}
 corresponds to removing the circle between $\Sigma_i$ and $\Sigma_{i+1}$. To $Y$ we associate the diagram $\mathcal{D}(Y)$: 
\[
\begin{tikzcd} 
 & & & . \ar[Rightarrow,d,"L", start anchor={[xshift=0ex, yshift=-2ex]}, end anchor={[xshift=0ex, yshift=-2ex]}] \ar[rd, bend left=20, "\N(\Sigma_{i+1})"]& & & \\
 . \ar[r, "\N(\Sigma_{1})"] &  \cdots \ar[r] & . \ar[ru, bend left=20, , "\N(\Sigma_{i})"]  \ar[rr,bend right=40,"\N(\Sigma_{i}\cup\Sigma_{i+1})"{name=H,swap}]   &  \  & . \ar[r] & \cdots  \ar[r, "\N(\Sigma_{k})"] & .
\end{tikzcd},
\]
where $L$ is the circle removal correspondence defined above (and  corresponding to the identification 2-morphism).

\vspace{.3cm}
\paragraph{\underline{General 2-handle attachmnents}}

  Assume now that $Y\colon \underline{\Sigma} \Rightarrow \underline{S}$, is a compression body, and let  
\[
\underline{M} = \xymatrix{ G_0 \ar[r]^{M_1}& G_1 \ar[r]^{M_2}& \cdots \ar[r]^{M_k}& G_k. }
\]  
\[
\underline{N} = \xymatrix{ G_0 \ar[r]^{N_1}& G_1 \ar[r]^{N_2}& \cdots \ar[r]^{N_k}& G_k. }
\] 
be the sequences corresponding respectively to $\underline{\Sigma}$ and $\underline{S}$. Associate then to $Y$ the following diagram $\mathcal{D}(Y)$: 
 \[
\xymatrix{
   G_0 \rtwocell^{M_1}_{N_1}{L_1} & G_1 \rtwocell^{M_2}_{N_2}{L_2} & G_2 \rtwocell^{}_{}{} &\cdots \rtwocell^{}_{}{} & G_{k-1} \rtwocell^{M_k}_{N_k}{L_k} & G_k
},
\]
where $L_i = L(\widetilde{Y}_i)$ stands for the correspondence of Definition~\ref{def:lagcorr_3mfds}, with $Y_i\subset Y$ the piece flowing from $\Sigma_i$ to $S_i$, and $\widetilde{Y}_i$ obtained by modding out  the vertical tubes (i.e. the tubes flowing between the circles of the decompositions $\underline{\Sigma}$ and $\underline{S}$) by the gradient flow of the Morse function, i.e. collapsing them to circles.
\begin{lemma}(Independence of functions/pseudo-gradient)
\label{lem:indep_functions_pseudograd}
The diagram  $\mathcal{D}(Y)$, as a 2-morphism of $\Hamhat$, is independent on the choice of Morse functions and pseudo-gradients.
\end{lemma}

\begin{proof}
First, the choice of a different pseudo-gradient can have the effect of twisting the vertical tubes, which might have the effect of changing the correspondences $L_i$, but not the whole diagram, as can be shown by a similar argument as in Lemma~\ref{lem:indep_param_1mph}.

Observe then that since they all are of the same index, the number of critical points on each piece $Y_i$ is determined by the decomposition of $\Sigma$ and $S$.

Notice finally that the Lagrangian correspondence only depend on the attaching circles. And they change just by isotopies or handleslides, which has no effect on $L_i$.

\end{proof}

\subsection{Cerf moves invariance}
\label{ssec:Cerf_moves_invariance}

We now prove that two decompositions of a given cobordism yield equivalent morphisms of $\Hamhat$, in the sense of Definition~\ref{def:equiv_rel_hamhat}. Before starting, we point out that the weak transversality assumption will always be satisfied in our case.  Indeed, $L$ and $(\prod{\underline{P}}) \times C$ will always be irreducible real algebraic affine varieties, therefore the only way that $L$ could fail to intersect $(\prod{\underline{P}}) \times C$ weakly transversely would be to be entirely contained in it, which is never the case, since in the holonomy descriptions, $L$ is defined by equations such as $\mathrm{Hol}_\gamma A = 1$ for some curves $\gamma$, while $C$ is defined by equations such as $\mathrm{Hol}_\gamma A = -1$: the trivial representation is always contained in $L$, but not in $C$.

At the level of 1-morphisms, the circle insertion/removal invariance follows from Proposition~\ref{prop:gluing_equals_reduction} and the definition of the equivalence relation.

We now check the moves for 2-morphisms.

\vspace{.3cm}
\paragraph{\underline{Diffeomorphism equivalence}} Follows from the fact that a diffeomorphism of surfaces induces a symplectomorphism on moduli spaces, that preserves the Hamiltonian actions, and a diffeomorphism of 3-manifolds maps Lagrangian correspondences to Lagrangian correspondences.

\vspace{.3cm}
\paragraph{\underline{Cylinder creation/cancellation}} Follows from the fact that the diagram associated to a cylinder is an identity in $\Hamhat$.

\vspace{.3cm}
\paragraph{\underline{Circle insertion/removal}} Follows from the definition of the equivalence relation in $\Hamhat$ (Definition~\ref{def:equiv_rel_hamhat}), and the fact that the diagram associated to a circle removal is an identification 2-morphism, modulo this relation.

\vspace{.3cm}
\paragraph{\underline{Imbrication of compression bodies}} Follows immediately from Proposition~\ref{prop:compo_fibered_coisotropics}.

\vspace{.3cm}
\paragraph{\underline{Critical point switches}} There are several cases to be distinguished, depending on the index of the critical points.

When both index are equal to either 1 or 2, this is a special case of imbrication of compression bodies. 

When one of the critical points has index either 0 or 3, its attaching sphere (or attaching belt) is a sphere (with some circles in it) and in particular a whole connected component of the total space of $\underline{\Sigma}$. It follows that the diagrams associated to the handle attachments correspond to attaching  2-cells to the sequence $\underline{M}$ associated to $\underline{S}$ in a way that does not overlap. Therefore the order of attachment doesn't matter.

It remains to check the case when the indexes are 1 and 2.  

Let $S_1, S_2 \subset \underline{\Sigma}$ denote respectively a 0-{sphere} and a 1-sphere, disjoint from each other. If $S_1$ and $S_2$ lie in different components of the sequence $\underline{\Sigma}$, invariance follows from the same reasons as in the previous case.
If this is not the case, one can insert circles separating them  (by circle insertion invariance), so that they lie on distinct components.

\vspace{.3cm}
\paragraph{\underline{Index 0-1 (and 2-3) handle cancellation}} Let $(Y_i, Y_{i+1}, Y_{i+2})$ be as in Proposition~\ref{prop:Cerf_moves}. First, $\N(D^0)\simeq\N(D^1)\simeq pt$, and the diagram associated to $Y_i$ consists in inserting the correspondence $pt\subset \N(D^0)\times \N(D^1)$. Observe then that $\N(\Sigma_i^1)$ and $\N(D^0 \cup (D^1 \sharp \Sigma_i^1))$ are identified, and both correspond to the symplectic quotient of $\N(D^1 \sharp \Sigma_i^1)$, by the $SU(2)$-action on $\partial D^1$. The diagrams associated to $Y_{i+1}$ and $Y_{i+2}$ both correspond (modulo the equivalence relation of Definition~\ref{def:equiv_rel_hamhat}) to inserting the identification 2-morphism for this reduction. It follows that the diagram associated to $(Y_i, Y_{i+1}, Y_{i+2})$ is an identity, modulo the equivalence relation. 

Reversing the diagram, we obtain the index 2-3 cancellation move.

\vspace{.3cm}
\paragraph{\underline{Index 1-2 handle cancellation}}
Assume that $Y_i$ and $Y_{i+1}$ correspond to a handle cancellation pair, i.e. (the opposite of) $Y_i$ and $Y_{i+1}$ correspond to 2-handle attachments along two closed curves that intersect transversely  at a single point. We can assume that these curves correspond respectively to $\alpha_1$ and $\beta_1$, with $\alpha_1$,  $\beta_1$, ..., $\alpha_g$, $\beta_g$ a symplectic basis of the fundamental group of the intermediate surface $\Sigma$ (with the $\gamma$ curves removed). 

These two curves define two coisotropic submanifolds of $\N(\Sigma)$ 
\begin{align*}
C_1 &=\lbrace A_1 =1\rbrace\text{, and} \\
C_2 &=\lbrace B_1 =1\rbrace 
\end{align*}
that induce respectively $L(Y_1)^T$ and $L(Y_2)$. Since they intersect transversely, it follows that their composition is embedded, and corresponds to $L(Y_i\cup Y_{i+1})$, i.e. the diagonal correspondence.

\begin{remark}\label{rem:MooreTachikawa}
At dimensions $(1+1)$, our construction is formally similar with Moore-Tachikawa's TQFTs \cite{MooreTachikawa}. Here are the main differences between the two constructions:
\begin{itemize}
\item In \cite{MooreTachikawa}, objects in the target category $HS$ are complex algebraic groups, and 1-morphisms are holomorphic symplectic varieties with !hamiltonian actions, while in $\Ham$ we consider compact real Lie groups and smooth (real) symplectic manifolds. Notice that some versions of instanton homology for $SL(2, \cc)$ have been defined in \cite{AbouzaidManolescu,CoteManolescu}, for these versions one can expect a similar $(1+1+1)$-TQFT structure taking values in a category closer to $HS$.
\item One of the requirements in \cite{MooreTachikawa} is a strict glueing equals reduction formula for glueing surfaces, while in our construction such a formula only holds up to a codimension 3 submanifold.
\item Finally, in \cite{MooreTachikawa} it is required that the 2-disc is sent to $G\times \mathcal{S}$, with $\mathcal{S}$ a slodowy slice, and that the cylinder $S^1\times [0,1]$ is sent to $T^*G$. This is not the case in our construction, as these surfaces are sent respectively to the point, and to an open subset of $T^*G$.
\end{itemize}
\end{remark}

\section{Future directions}\label{sec:future_dir}

We now outline some directions we plan to take in the future. 

\subsection{Existence of a (non-quasi) 2-functor, quasi-Hamiltonian analogue}\label{ssec:qHam}

By using the spaces $\N(\Sigma)$, we were able to construct a quasi 2-functor. A natural question follows:

\begin{question} Is it possible to replace the spaces $\N(\Sigma)$ in our construction by suitable Hamiltonian manifolds, satisfying a gluing equals reduction principle in a strict sense (i.e. without $C$ in Proposition~\ref{prop:gluing_equals_reduction}), so to obtain a 2-functor $\Cob_{1+1+1} \to \Hamhat$ that still assigns to a closed surface a sequence whose composition is the $SU(2)$-character variety?
\end{question}

We believe this question can be answered partially positively, at least in two ways. First, by using the quasi-Hamiltonian spaces constructed in \cite{AMM_q-ham}. These spaces satisfy a gluing equals reduction principle, but their moment map takes values in the group, rather than its Lie algebra. This would lead to the definition of an analogous partial 2-category $\qHam$. However, it seems not obvious to define Floer homology in this setting, since these are not symplectic manifolds.

Another possible solution would be to use the infinite dimensional moduli spaces $\M(\Sigma)$ introduced by Donaldson \cite{DonaldsonBoundary_value}: these are moduli spaces of flat {connections} on $\Sigma$, but the restriction to the boundary may not be constant, and defines a map $A_{|\partial \Sigma}\colon \partial \Sigma \to \mathfrak{g}$. This moduli space is acted on by the gauge group of $\partial \Sigma$, which identifies with $k$ copies of the loop group $LSU(2)$.  

\subsection{Extension to dimension zero}\label{ssec:dim_zero} Since the former spaces $\M(\Sigma)$ are infinite dimensional, Floer homology seems also difficult to define in this setting. However, it seems possible to use these spaces in order to define a theory extended to dimension zero: to a closed interval $I$ one can associate the path group $\G(I) = Map(I, SU(2))$, which comes with an evaluation map to the boundary of $I$: $ev\colon \G(I) \to SU(2)^2$. One can use these evaluation maps to glue the groups: if $S^1 = I\cup J$ is a decomposition of the circle into two intervals, one has 
\[
\G(S^1) = \G(I) \times_{ev} \G(J).
\]
This suggests the definition of a partial 3-category, where objects would be finite dimensional Lie groups, 1-morphisms Banach Lie groups together with morphisms similar with these evaluation maps, 2-morphisms Banach Hamiltonian manifolds, and 3-morphisms $\G$-Lagrangian correspondences.

\subsection{Extension to dimension four}\label{ssec:dim_four} The aim of this project is to promote $\Hamhat$, or at least the pre-completion $\underline{\Ham}$, to a 3-category, and extend the quasi 2-functor defined here to a quasi 3-functor from $\Cob_{1+1+1+1}$ (or a version enriched with cohomology classes).

If $G$ and $G'$ are objects of $\Ham$, $\underline{M},\underline{M'} \in \underline{hom}^1(G,G')$, and $\mathcal{D},\mathcal{D}'\in  \underline{hom}^2(\underline{M},\underline{M'})$, one would like to define a 3-morphism space $hom^3(\mathcal{D},\mathcal{D}')$ using equivariant Floer homology (or the chain complex defining it). Indeed, in such a situation, let
\begin{align*}
 \M &=\prod_{M}{M^-\times M}, \\
 \mathscr{L}_0&=\prod_{M}{\Delta_M}, \\
 \mathscr{L}_1&=\prod_{L}{L}, \\
 \G&=\prod_{G}{G} ,
\end{align*}
where $M$, $L$ and $G$ run respectively in the set of symplectic manifolds, Lagrangian multi-correspondences and  Lie groups appearing in the diagram $\mathcal{D} \sharp_{\underline{M}\sharp\underline{M'}}\mathcal{D}'$. One would then take 
\[
hom^3(\mathcal{D},\mathcal{D}') = CF_\G (\mathscr{L}_0, \mathscr{L}_1).
\]
Several constructions of such chain complexes appeared in the literature, for example \cite{Frauenfelder,HLSlie}, or a construction outlined in \cite{DaemiFukaya}. We plan to define another version relying on \WW's quilt theory, that should be well suited for our purposes. The resulting algebraic structure should be a 3-category analogue of Bottman and Carmeli's $(A_\infty,2)$-categories \cite{BottmanCarmeli}.

\subsection{Invariants for knots and sutured manifolds}\label{ssec:knots_sutures}

The framework developped here should be well-suited for defining invariants of knots, and more generally sutured manifolds, similar with the ones in Heegaard-Floer theory, since a sutured manifold can be viewed as a 2-morphism in $\Cob_{1+1+1}$. After applying the functor, one gets a 2-morphism in $\Hamhat$, which can be horizontally composed with the coadjoint orbit $\left\lbrace \theta \in \mathfrak{g}\ |\ \abs{\theta} = \frac{\pi}{2} \right\rbrace$. That precisely corresponds to putting a traceless condition on the holonomy of a {connection} around a meridian of a knot, as in for example \cite{KMknots}.

\subsection{Relation with Seiberg-Witten theory}\label{ssec:rel_SW}
The ideas in this section emerged during a conversation with Guangbo Xu. Denote 
\[
F_{Don}\colon \Cob_{1+1+1} \dashrightarrow \Hamhat
\]
the quasi 2-functor defined in this paper. One should be able to define a similar quasi 2-functor 
\[
F_{SW}\colon \Cob_{1+1+1} \dashrightarrow \Hamhat
\]
 that would correspond to Seiberg-Witten theory, using extended moduli spaces of vortices analogous to the spaces $\N(\Sigma)$. 

Following Witten's conjecture, these two theories should be related: one can expect that there is a natural transformation 
\[
T\colon F_{Don} \to F_{SW},
\]
i.e. to any $k$-morphism $W\colon X\to Y$ in $\Cob_{1+1+1}$, should correspond a $(k+1)$-morphism in $\Hamhat$ relating $F_{Don}(W)$ and $F_{SW}(W)$ as in the following diagram: 
\[
\xymatrix{ F_{Don}(X)\ar[dd]^{T(X)} \ar[rr]^{F_{Don}(W)} & & F_{Don}(Y) \ar[dd]^{T(Y)} \\
& \Downarrow T(W) & \\
F_{SW}(X) \ar[rr]^{F_{SW}(W)} & & F_{SW}(Y) .
}
\]

One can think of $T(W)$ as being associated with the $(k+1)$-morphism $W\times [0,\infty]$, where the $[0,\infty]$ coordinate could correspond to a wavelength. We plan to build $F_{SW}$ and $T$ in a future joint work, building on the work of Feehan and Leness \cite{FeehanLeness1,FeehanLeness2}.

\subsection{Extension to higher dimensional gauge theory}\label{ssec:rel_Haydys}

In \cite{DonaldsonThomas}, Donaldson and Thomas pave the way towards generalizations of Donaldson-Floer theory in dimensions 2,3 and 4 to higher dimensions: one in complex geometry, another one for manifolds with special holonomy. 
Their starting point is the so called Donaldson-Thomas invariant of a Calabi-Yau 3-fold,  a holomorphic analogue of the Casson invariant (which is the Euler characteristics of Instanton Homology). These new theories share some formal similarities with the low-dimensional one. However, some new serious technical difficulties arise, making a full implementation of it a hard challenge. In particular, compactness problems are considerably more delicate.

Nevertheless, for the complex geometry generalization, this theory was successfully implemented via algebraic geometry \cite{Thomas,JoyceSong,KontsevichSoibelman}, in algebraic settings. In the special holonomy setting, such techniques are not available, therefore a Floer theory would be desirable.  In \cite{DonaldsonSegal}, Donaldson and Segal lay out some foundations towards  such a theory. In particular, they suggest (in \cite[Section 3.3]{DonaldsonSegal})  that to a Calabi-Yau 3-fold should correspond a Fukaya category associated with a certain moduli space of Hermite-Einstein connections.  These moduli spaces are not smooth in general, but are locally critical loci \cite{JoyceSong}.

Haydys \cite{Haydys} refined the expectations in the case when the 6-manifold corresponds to a twisted spinor bundle over a Riemannian 3-manifold: he posits that the right object to consider would be a Fukaya-Seidel category associated to the moduli space of complexified connections, endowed with the complex Chern-Simons functional. Moreover, he outlines what could then be a corresponding Field theory in dimensions (3+1+1) in this setting, involving the Vafa-Witten and the Haydys-Witten equations in dimensions 4 and 5 respectively. He does so by giving a new definition of Fukaya-Seidel categories that makes the relation with these equations more transparent. Progress in a similar direction has been made by Wang \cite{wang2021monopoles} in the Seiberg-Witten setting.

Then, a natural question is what algebraic structures would one observe when doing a further dimensional reduction (which, at the level of equations, corresponds to an analog of Nahm's equations). Would one land in a category similar to $\Ham$ (where maybe 1-morphisms would consist in Lefschetz fibrations with fiber-preserving Hamiltonian actions)? In other words, can one associate  to 3-manifolds with boundary some moduli spaces that would play a role similar to the extended moduli spaces $\N(\Sigma)$? Equipped with some Hamiltonian group actions and satisfying some glueing equals reduction principle? At least for simple enough elementary 3-cobordisms, would these be better behaved than those associated with a closed 3-manifold? Would it then be possible to adapt Haydys construction of the Fukaya-Seidel category to an equivariant setting?

\bibliographystyle{alpha}
\bibliography{biblio}

\end{document}